\documentclass[numbers,webpdf,imaiai]{ima-authoring-template}

\graphicspath{{Fig/}}

\theoremstyle{thmstyletwo}%
\newtheorem{theorem}{Theorem}
\newtheorem{proposition}[theorem]{Proposition}%
\newtheorem{remark}{Remark}%

\newtheorem{definition}{Definition}

\newtheorem{lemma}[theorem]{Lemma}
\newtheorem{corollary}[theorem]{Corollary}

\numberwithin{equation}{section}

\usepackage{mathtools}
\mathtoolsset{centercolon}  

\numberwithin{equation}{section} 

\definecolor{mygreen}{rgb}{0.1,0.75,0.2}

\providecommand{\mathbbm}{\mathbb} 

\newcommand{\R}{\mathbbm{R}}

\newcommand{\F}{\mathcal{F}}

\renewcommand{\phi}{\varphi}

\usepackage[font = small, margin=30pt]{caption}

\usepackage[]{algorithm}
\usepackage{algpseudocode}
\usepackage{enumerate}

\usepackage{graphicx}
\usepackage[dvipsnames]{xcolor}

 \counterwithin{theorem}{section}
  \counterwithin{definition}{section}
  \counterwithin{remark}{section}

\usepackage{comment}

\newcommand{\tdiam}{\text{diam}}

\newcommand{\iid}{\stackrel{\text{i.i.d.}}{\sim}}

\newcommand{\E}{\mathbb{E}}

\renewcommand{\P}{\mathbb{P}}

\newcommand{\mcF}{\mathcal{F}}

\newcommand{\mcN}{\mathcal{N}}

\usepackage{multirow}

\newcommand{\inparen}[1]{\left(#1\right)}             
\newcommand{\insquare}[1]{\left[#1\right]}             

\usepackage[scr = esstix, cal = cm, frak=euler]{mathalfa}

\definecolor{mygreen}{rgb}{0.1,0.75,0.2}

\begin{document}

\DOI{DOI HERE}
\copyrightyear{2021}
\vol{00}
\pubyear{2021}
\access{Advance Access Publication Date: Day Month Year}
\appnotes{Paper}
\copyrightstatement{Published by Oxford University Press on behalf of the Institute of Mathematics and its Applications. All rights reserved.}
\firstpage{1}



\title[Sharp concentration of simple random tensors]{Sharp concentration of simple random tensors}


\author{Omar Al-Ghattas
\address{\orgdiv{Department of Statistics}, \orgname{University of Chicago}, \orgaddress{ \postcode{IL 60637}, \country{USA}}}}
\author{Jiaheng Chen*
\address{\orgdiv{Committee on Computational and Applied Mathematics}, \orgname{University of Chicago}, \orgaddress{ \postcode{IL 60637}, \country{USA}}}}
\author{Daniel Sanz-Alonso
\address{\orgdiv{Department of Statistics}, \orgname{University of Chicago}, \orgaddress{ \postcode{IL 60637}, \country{USA}}}}


\authormark{O. Al-Ghattas, J. Chen, D. Sanz-Alonso}

\corresp[*]{Corresponding author: \href{email:jiaheng@uchicago.edu}{jiaheng@uchicago.edu}}

\received{Date}{0}{Year}
\revised{Date}{0}{Year}
\accepted{Date}{0}{Year}


\abstract{
This paper establishes sharp dimension-free concentration inequalities and expectation bounds for the deviation of the sum of simple random tensors from its expectation. As part of our analysis, we use generic chaining techniques to obtain a sharp high-probability upper bound on the suprema of $L_p$ empirical processes. In so doing, we generalize classical results for quadratic and product empirical processes to higher-order settings.
}



\maketitle

\section{Introduction}\label{sec:introduction}
 This paper establishes sharp bounds for the operator-norm deviation of the sum of simple (rank-one) random tensors from its expectation. Let $X,X_1,\ldots,X_N$ be i.i.d. centered Gaussian random variables in a separable Hilbert space $H$ with covariance operator $\Sigma$.
 Our first main result, Theorem \ref{thm:main1}, shows that, with a standard notation described below,  for any integer $p\ge 2$, 
 \begin{equation*}
 \E \bigg\|\frac{1}{N}\sum_{i=1}^{N} X_i^{\otimes p}-\E\, X^{\otimes p}\bigg\|\asymp_p \|\Sigma\|^{p/2} \bigg(\sqrt{\frac{r(\Sigma)}{N}}+\frac{r(\Sigma)^{p/2}}{N}\bigg),\quad r(\Sigma):=\frac{\mathrm{Tr}(\Sigma)}{\|\Sigma\|}.
 \end{equation*}
 In addition, Theorem \ref{thm:main1} establishes sharp high-probability and expectation bounds for sub-Gaussian data. To prove Theorem \ref{thm:main1}, we analyze the suprema of a $L_p$ empirical process. In our second main result, Theorem \ref{thm:main2}, we leverage generic chaining techniques to obtain a sharp high-probability bound on   
 \begin{equation*}
 \sup_{f\in \F} \bigg|\frac{1}{N}\sum_{i=1}^N f^p(X_i) -\E f^p(X)\bigg|, \qquad p \ge 2,
\end{equation*}
in terms of quantities that reflect the geometric complexity of the family $\F,$ taken to contain bounded linear functionals in $H$ in the proof of Theorem \ref{thm:main1}. 

\subsection{Outline}
Section \ref{sec:main} states  our two main results and discusses related work. The study of concentration of simple random tensors and $L_p$ empirical processes is carried out in Sections \ref{sec:concentration} and \ref{sec:multi}, respectively. Section \ref{sec:conclusions} closes with conclusions and future directions.

\subsection{Notation}\label{sec:notation}
Given two positive sequences $\{a_n\}$ and $\{b_n\}$, we write $a_n\lesssim b_n$ to denote that $a_n \le c b_n$ for some absolute constant $c>0$. If both $a_n \lesssim b_n$ and $b_n \lesssim a_n$ hold simultaneously, we write $a_n \asymp b_n$. If the constant $c$ depends on some parameter $\tau$, we use the notations $a_n \lesssim_{\tau} b_n,b_n\lesssim_{\tau} a_n$, and $a_n\asymp_{\tau} b_n$ to indicate this dependence.

\section{Main results}\label{sec:main}
Here we state the two main theorems of the paper and compare them with existing results. 
Subsection \ref{subsec:cencentraion_of_tensor} focuses on the concentration of simple random tensors and
Subsection \ref{subsec:L_p} focuses on $L_p$ empirical processes.

\subsection{Concentration of simple random tensors}\label{subsec:cencentraion_of_tensor}

Let $(H, \langle \cdot, \cdot\rangle_H, \| \cdot \|_H)$ be a separable real Hilbert space and let $p$ be a positive integer. Let $X$ be a centered $H$-valued random variable. Suppose that $\E|\langle X, v \rangle_H|^p < \infty$ for all $v \in U_H:= \{v \in H : \|v\|_H = 1 \}$ and let $X^{\otimes p}$ be the multilinear form
\begin{equation*}
    X^{\otimes p} (v_1, \ldots, v_p) := \langle X,v_1 \rangle_H \cdots \langle X, v_p \rangle_H, \qquad \forall \ (v_1, \ldots, v_p) \in H \times \cdots \times H.
\end{equation*}
Thus, $X^{\otimes p}$ is a random variable taking values in the tensor space $H^{\otimes p}.$ We refer for instance to \cite[Section II.4]{reed1980methods} for background on tensor products of Hilbert spaces. The expectation $\E\, X^{\otimes p} $ is called the \emph{$p$-th order moment tensor} of $X.$ When $p=2,$ the second-order moment tensor $\E\, X^{\otimes 2}(v_1,v_2) = \E \langle X, v_1\rangle_H \langle X,v_2\rangle_H,$  $\, \forall \, (v_1, v_2) \in H \times H $ is called the \emph{covariance} of $X.$ By the Riesz representation theorem, this bilinear form can be represented by the \emph{covariance operator} $\Sigma: H \to H$ defined by 
$$ \langle v_1, \Sigma v_2 \rangle_H : =  \E \,X^{\otimes 2}(v_1,v_2), \qquad    \forall \ (v_1,v_2) \in H \times H.$$

 Given independent copies $X_1, \ldots, X_N$ of $X,$ our goal is to establish concentration inequalities for the $p$-th order sample moment tensor $\frac{1}{N} \sum_{i=1}^N X_i^{\otimes p},$ which is a sum of i.i.d. simple (rank-one) random tensors.   Specifically, we seek to obtain high-probability and expectation bounds on the deviation 
\begin{equation}\label{eq:opnormdeviation}
 \bigg\|\frac{1}{N}\sum_{i=1}^{N} X_i^{\otimes p}-\E \, X^{\otimes p}\bigg\| = \sup_{v_1, \ldots, v_p \in U_H} \bigg| \frac{1}{N}\sum_{i=1}^N \langle X_i,v_1\rangle_H \cdots \langle X_i, v_p \rangle_H  -  \E\langle X,v_1 \rangle_H \cdots \langle X, v_p \rangle_H \bigg|,
 \end{equation}
where $\| \cdot \|$ denotes the norm in the tensor space $H^{\otimes p}.$ Notice that for $p=2,$ the operator norm of the covariance operator $\Sigma: H \to H$ agrees with the norm of the covariance  $\E \,X^{\otimes 2}$ as a bilinear form in $H^{\otimes 2}.$ Consequently, we will slightly abuse notation and  denote by $\| \Sigma\|$ the operator norm of the covariance operator. Furthermore, to simplify the notation, we will henceforth omit the subscript $H$ in the inner-product $\langle \cdot, \cdot \rangle_H$ and norm $\| \cdot \|_H$. 

We will focus on Gaussian and sub-Gaussian random variables. Recall that an $H$-valued random variable $X$ is called \emph{Gaussian} iff all one-dimensional projections $\langle X, v \rangle$ for $v \in H$ are Gaussian real-valued random variables. A centered random variable $X$ in $H$ is called \emph{sub-Gaussian} iff
    \[
    \|\langle X,v \rangle\|_{\psi_2}\lesssim  \|\langle X,v \rangle\|_{L_2}, \qquad \forall \ v \in H,
    \]
where the Orlicz $\psi_2$-norm of a real-valued random variable $Y$ is defined as $\|Y\|_{\psi_2}:=\inf \left\{c>0: 
\E \exp (Y^2/c^2) \leq 2\right\}$. Furthermore, $X$ is called \emph{pre-Gaussian} iff there exists a centered Gaussian random variable $Y$ in $H$ with the same covariance operator as that of $X.$

Our first main result, Theorem \ref{thm:main1}, provides sharp dimension-free bounds for the deviation \eqref{eq:opnormdeviation} in terms of the    \emph{effective rank}  $r(\Sigma): = \text{Tr}(\Sigma)/\| \Sigma\|$ of the covariance operator.

\begin{theorem}\label{thm:main1}
Let $X, X_1, \ldots, X_N$ be i.i.d. centered sub-Gaussian and pre-Gaussian random variables in $H$ with covariance operator $\Sigma$. For any integer $p\ge 2$ and any $u\ge 1$, it holds with probability at least $1-\exp(-u)$ that
\[
\bigg\|\frac{1}{N}\sum_{i=1}^{N} X_i^{\otimes p}-\E\, X^{\otimes p}\bigg\| \lesssim_p \|\Sigma\|^{p/2} \bigg(\sqrt{\frac{r(\Sigma)}{N}}+ \frac{r(\Sigma)^{p/2}}{N}+\sqrt{\frac{u}{N}}+ \frac{u^{p/2}}{N}\bigg).
\]
As a corollary,
 \begin{align*}
\E \bigg\|\frac{1}{N}\sum_{i=1}^{N} X_i^{\otimes p}-\E\, X^{\otimes p}\bigg\| \lesssim_p \|\Sigma\|^{p/2} \bigg(\sqrt{\frac{r(\Sigma)}{N}}+\frac{r(\Sigma)^{p/2}}{N}\bigg).
 \end{align*}
Moreover, if $X$ is Gaussian, then
\begin{align*}
\E \bigg\|\frac{1}{N}\sum_{i=1}^{N} X_i^{\otimes p}-\E\, X^{\otimes p}\bigg\| \asymp_p \|\Sigma\|^{p/2} \bigg(\sqrt{\frac{r(\Sigma)}{N}}+\frac{r(\Sigma)^{p/2}}{N}\bigg).
\end{align*} 
\end{theorem}

The proof of Theorem \ref{thm:main1} is carried out in Section \ref{sec:concentration}.
Although upper bounds for \eqref{eq:opnormdeviation} have been extensively studied \cite{giannopoulos2000concentration,guedon2007lp, mendelson2008weakly, adamczak2010quantitative, vershynin2011approximating, mendelson2021approximating,even2021concentration,zhivotovskiy2024dimension}, and despite the explosive growth of research on concentration inequalities for random tensors \cite{vershynin2020concentration,zhou2021sparse,even2021concentration,bamberger2022hanson,jiang2022near,boedihardjo2024injective,bandeira2023matrix,bandeira2024matrix,santos2023almost,bandeira2024geometric}, Theorem \ref{thm:main1} is, to the best of our knowledge, the first to establish \emph{sharp dimension-free} concentration inequalities and expectation bounds for the deviation \eqref{eq:opnormdeviation} under sub-Gaussian assumptions, without superfluous logarithmic factors. Moreover, Theorem \ref{thm:main1} shows that our expectation bound is sharp for Gaussian data. Similarly to \cite{koltchinskii2017concentration}, our results can easily be extended to Banach-valued random variables, but we focus on the Hilbert space setting for ease of exposition.

For sub-Gaussian and pre-Gaussian random variables satisfying that, for fixed $K>0$,
\begin{equation}\label{eq:KsubG}
    \|\langle X,v \rangle\|_{\psi_2}\le K  \|\langle X,v \rangle\|_{L_2}, \qquad \forall \ v\in H,
\end{equation}
a straightforward modification of the proof of Theorem \ref{thm:main1} yields
\[
\E \bigg\|\frac{1}{N}\sum_{i=1}^{N} X_i^{\otimes p}-\E\, X^{\otimes p}\bigg\| \lesssim_p K^p \|\Sigma\|^{p/2} \bigg(\sqrt{\frac{r(\Sigma)}{N}}+\frac{r(\Sigma)^{p/2}}{N}\bigg),
\]
as well as a similar high-probability upper bound. For a centered Gaussian random variable, \eqref{eq:KsubG} holds with $K = \sqrt{8/3}$. When $p=2$, the dependence on $K^2$ is consistent with the literature on covariance estimation; see \cite[Theorem 4.7.1]{vershynin2018high}. For $p> 2$, the dependence on $K^p$ agrees with the result of  \cite{zhivotovskiy2024dimension}. The notation $\lesssim_p$ indicates that the implicit constant depends on the tensor order $p$, and this dependence is generally super-exponential in $p$. This can be seen from the fact that, when $p$ is even and $X\sim \mcN(0,\Sigma)$,
\[
\|\E\, X^{\otimes p}\|= (p-1)!!\|\Sigma\|^{p/2},\quad \text{ and } \quad (p-1)!!\asymp \sqrt{2 (p-1)}\Big(\frac{p-1}{e}\Big)^{\frac{p-1}{2}}\asymp e^{\frac{1}{2}p \log p}.
\]
In Theorem \ref{thm:main1} and in the comparisons that we will provide in the remainder of this subsection, we assume $K\asymp 1$ for simplicity. In Remark \ref{rem:compare_1}, we compare with concentration results for sums of simple tensors, while in Remark \ref{rem:discuss_sample_cov} we discuss the special case $p=2.$

\begin{remark}\label{rem:compare_1}

The deviation \eqref{eq:opnormdeviation} can be expressed as
\begin{align}\label{eq:compare_aux0}
\bigg\|\frac{1}{N}\sum_{i=1}^{N} X_i^{\otimes p}-\E\, X^{\otimes p}\bigg\|=\sup _{\|v\|= 1}\bigg|\frac{1}{N} \sum_{i=1}^N\left\langle X_i, v\right\rangle^p-\E\langle X, v\rangle^p\bigg|,
\end{align}
since for symmetric tensors the supremum in \eqref{eq:opnormdeviation} is attained by a single element $v\in H$, see e.g. \cite[(1.6)]{zhivotovskiy2024dimension}, \cite[Section 2.3]{nemirovski2004interior} and \cite[Remark 4.14, Proposition 4.15]{bandeira2024geometric}. From our proof of Theorem \ref{thm:main1}, the same high-probability and in-expectation bounds in Theorem \ref{thm:main1} also hold for $
   \sup _{\|v\|= 1}\left |\frac{1}{N} \sum_{i=1}^N\left |\langle X_i, v\right\rangle|^p-\E|\langle X, v\rangle|^p\right|$, where $\langle X_i,v \rangle^p$ is replaced by its absolute value.

Using majorizing measures,  \cite[Theorem 3]{guedon2007lp} establishes the bound
\begin{align}\label{eq:compare_aux1}
\E \sup_{\|v\|\le 1}\bigg|\frac{1}{N} \sum_{i=1}^N | \langle X_i, v\rangle|^p-\E|\langle X, v\rangle|^p\bigg|\lesssim_p \|\Sigma\|^{p/2}\left(\sqrt{\varepsilon}+\varepsilon\right),
\end{align}
where 
\[
\varepsilon=\frac{\log N}{N}\cdot \frac{\E \max_{1\le i\le N}\|X_i\|^p}{\|\Sigma\|^{p/2}}.
\]
Notably, \cite[Theorem 3]{guedon2007lp} considered the more general case where the supremum in \eqref{eq:compare_aux1} is taken over a symmetric convex body $K\subset \R^d$. Here, we focus on the specific case where $K$ is the unit ball $\{v:\|v\|\le 1\}$. In the sub-Gaussian setting, applying \cite[Lemma 2.7]{zhivotovskiy2024dimension} together with a union bound gives
\[
\bigg(\mathbb{E} \max _{1 \leq i \leq N}\left\|X_i\right\|^p\bigg)^{1/p} \lesssim_p \sqrt{\mathrm{Tr}(\Sigma)+ \|\Sigma\|\log N}=\sqrt{\|\Sigma\| \left(r(\Sigma) +\log N\right)}.
\]
For Gaussian data, this upper bound is optimal up to a constant depending only on $p$. Consequently, this yields the following upper bound for the value of $\varepsilon$ in \eqref{eq:compare_aux1}:
\[
\varepsilon \lesssim_p \frac{(\log N)\left(r(\Sigma)^{p/2}+(\log N)^{p/2}\right)}{N}.
\]

Employing estimates of the metric entropy of ellipsoids, \cite[Theorem 3]{even2021concentration} shows that, for i.i.d. centered sub-Gaussian random vectors $X, X_1,\ldots,X_N$ in $\R^d$, 
\begin{equation}\label{eq:boundaux}
  \E \bigg\|\frac{1}{N}\sum_{i=1}^{N} X_i^{\otimes p}-\E\, X^{\otimes p}\bigg\|\lesssim_p   \|\Sigma\|^{p/2}\sqrt{\frac{(\log N)^p(r(\Sigma)+\log d)^{p+1}}{N}}.
\end{equation}
Clearly, the in-expectation bound  in Theorem \ref{thm:main1} provides a direct improvement over the bounds in \eqref{eq:compare_aux1} and \eqref{eq:boundaux}.

More recently, \cite{zhivotovskiy2024dimension} used the variational principle and the PAC-Bayesian method to study \eqref{eq:compare_aux0}. We now compare Theorem \ref{thm:main1} with \cite[Theorem 1.7]{zhivotovskiy2024dimension}. Under their assumption that $N \ge r(\Sigma)^{p-1}$, we observe that $\sqrt{\frac{r(\Sigma)}{N}}\ge \frac{r(\Sigma)^{p/2}}{N},$ and taking $u=r(\Sigma)$ in Theorem \ref{thm:main1} yields that with probability at least $1-\exp (-r(\Sigma))$,
\[
\bigg\|\frac{1}{N}\sum_{i=1}^{N} X_i^{\otimes p}-\E\, X^{\otimes p}\bigg\|=\sup_{\|v\|=1}\bigg|\frac{1}{N} \sum_{i=1}^N\left\langle X_i, v\right\rangle^p-\E\langle X, v\rangle^p\bigg|  \overset{\hspace{-0.1cm}\text{($\star$)}}{\lesssim_p} \|\Sigma\|^{p/2} \sqrt{\frac{r(\Sigma)}{N}}.
\]
For $H = \R^d$ and under the condition $N \ge r(\Sigma)^{p-1}$, \cite[Theorem 1.7]{zhivotovskiy2024dimension} establishes the bound ($\star$), but only with probability at least $1-cN\exp(-\sqrt{r(\Sigma)})$. 
Theorem \ref{thm:main1} hence directly improves \cite[Theorem 1.7]{zhivotovskiy2024dimension} for sub-Gaussian data. Moreover, 
our high-probability bound can be directly integrated to derive an in-expectation bound, which is not provided in \cite[Theorem 1.7]{zhivotovskiy2024dimension}. The proof of \cite[Theorem 1.7]{zhivotovskiy2024dimension} relies on a combination of the variational inequality approach and the decoupling-chaining method developed in \cite{adamczak2010quantitative, talagrand2022upper}. In the decoupling-chaining argument \cite[Proposition 14.3.3]{talagrand2022upper}, there is a crucial term: $\max_{1\le i\le N} \|X_i\|$. \cite{zhivotovskiy2024dimension} used a union bound to control this term, which results in a large prefactor $N$ in the probability estimate. Since for Gaussian data it holds that 
\[
\E \max_{1\le i\le N} \|X_i\|\asymp \sqrt{\|\Sigma\|(r(\Sigma)+\log N)},
\]
an additional logarithmic factor $\sqrt{\log N}$ is likely unavoidable when using the approach in \cite{zhivotovskiy2024dimension} to derive an in-expectation bound. Finally, in contrast to \cite[Theorem 1.7]{zhivotovskiy2024dimension}, our proof of Theorem \ref{thm:main1} shows a bound of the form ($\star$) for arbitrary $p \ge 2,$ not necessarily an integer.

\end{remark}

\begin{remark}\label{rem:discuss_sample_cov}
When $p=2$, the sum of simple random tensors reduces to the sample covariance. In this case, Theorem \ref{thm:main1} agrees with the optimal bounds established in \cite[Theorem 4, Theorem 9]{koltchinskii2017concentration}. 
An important area of research in probability and statistics focuses on the non-asymptotic analysis of sample covariance operators, particularly their estimation error in the spectral (operator) norm relative to its population version. 
Recent references include, for instance, \cite{rudelson1999random,vershynin2010introduction,tropp2015introduction,tropp2016expected,liaw2017simple,koltchinskii2017concentration,minsker2017some,van2017structured,vershynin2018high,tikhomirov2018sample,han2022exact,bandeira2023matrix,brailovskaya2024universality,bandeira2024matrix,zhivotovskiy2024dimension}. A key milestone  was achieved by Koltchinskii and Lounici \cite{koltchinskii2017concentration}, who established that for i.i.d. centered Gaussian random variables $X, X_1,\ldots,X_N$ with covariance operator $\Sigma$,
\begin{align}\label{eq:sample_cov}
 \E \bigg\|\frac{1}{N}\sum_{i=1}^{N} X_i \otimes X_i- \mathbb{E}\, X \otimes X\bigg\|\asymp \|\Sigma\| \bigg(\sqrt{\frac{r(\Sigma)}{N}}+\frac{r(\Sigma)}{N}\bigg).
\end{align}
The upper bound \eqref{eq:sample_cov} in \cite{koltchinskii2017concentration} follows from results on quadratic empirical processes developed in \cite{klartag2005empirical,mendelson2010empirical,dirksen2015tail,bednorz2014} that extend beyond Gaussian data. Alternative proofs exist for Gaussian data outside the empirical process framework. In particular, \cite{van2017structured} proved \eqref{eq:sample_cov} using decoupling and Slepian-Fernique comparison inequalities. More recently, \cite{han2022exact} obtained a sharp version of \eqref{eq:sample_cov} with optimal constants via the Gaussian min-max theorem. Beyond the operator norm, \cite{puchkin2025sharper} derived dimension-free bounds on the Frobenius distance between the sample covariance and the population covariance.
\end{remark}

\subsection{ $L_p$ empirical processes}\label{subsec:L_p}
This subsection contains the statement of our second main result, Theorem~\ref{thm:main2}, which gives a general upper bound on the suprema of $L_p$ empirical processes. Theorem~\ref{thm:main2} will be key to establish our upper bounds in Theorem \ref{thm:main1}.  Let $X, X_1, \ldots, X_N \iid \mu$ be a sequence of random variables on a probability space $(\Omega, \mu)$.  For $p \ge 1 $ (not necessarily an integer), we consider the centered $L_p$ empirical process indexed by a class $\mathcal{F}$ of functions on $(\Omega, \mu)$ given by 
\begin{align}\label{eq:L_p_empiricalprocess}
    f \mapsto 
    \frac{1}{N} \sum_{i=1}^N f^p(X_i) 
    - \E f^p(X),
    \qquad f \in \mcF.
\end{align}
In our proof of Theorem \ref{thm:main1}, the class $\F$ will be taken to contain all linear functionals in $H$ with norm one.  For $p = 1$, \eqref{eq:L_p_empiricalprocess} recovers the standard empirical process \cite{van2023weak, talagrand2022upper}, and for $p = 2$ the quadratic empirical process extensively studied, e.g., in \cite{rudelson1999random,klartag2005empirical,mendelson2007reconstruction,adamczak2010quantitative,mendelson2010empirical,mendelson2012generic,dirksen2015tail,bednorz2014,mendelson2016upper,bednorz2016bounds, koltchinskii2017concentration,han2022exact}. Beyond its theoretical significance, the quadratic empirical process plays a crucial role in various statistical applications, including covariance estimation as discussed in Remark \ref{rem:discuss_sample_cov}. Our goal is to bound
\[
\sup_{f\in \F} \bigg|\frac{1}{N}\sum_{i=1}^N f^p(X_i) -\E f^p(X)\bigg|
\]
for any $p\ge 2$ in terms of quantities that reflect the geometric complexity of the indexing class $\F.$  To that end,  for any function $f$ on $(\Omega, \mu)$, we introduce the Orlicz $\psi_2$-norm of $f$ given by
\[
\|f\|_{\psi_2}:=\inf \left\{c>0: 
\mathbb{E}_{X \sim \mu}\insquare{ \exp \inparen{\frac{|f(X)|^2}{c^2}}} \leq 2\right\},
\]
with the convention $\inf \emptyset = \infty.$ 
Our upper bound will depend on $d_{\psi_2}(\F) := \sup_{f \in \F}\|f\|_{\psi_2}$ and on Talagrand's functional $\gamma(\mcF, \psi_2)$, whose definition we now recall.

\begin{definition}[{Talagrand's $\gamma$ functional \cite{talagrand2022upper} }]\label{def:admissible_sequence_gamma_functional}
Let $(\F, d)$ be a metric space. An admissible sequence is an increasing sequence $(\F_s)_{s\ge 0} \subset \F$ which satisfies $\F_s \subset \F_{s+1}, \left|\F_0\right|=1, \left|\F_s\right| \leq 2^{2^s}$ for $s \ge 1,$ and $\bigcup_{s=0}^{\infty}\F_s$ is dense in $\F$. 
Let
\[
\gamma(\F, d) :=\inf \sup _{f \in \F} \sum_{s \geq 0} 2^{s /2} d\left(f, \F_s\right),
\]
where the infimum is taken over all admissible sequences. We write $\gamma(\mcF, \psi_2)$ when the distance on $\mathcal{\F}$ is induced by the $\psi_2$-norm.
\end{definition}

We are ready to state our second main result. Its proof can be found in Section \ref{sec:multi}.
\begin{theorem}\label{thm:main2} Assume $0\in \F$ or $\F$ is symmetric (i.e., $f\in \F$ implies $-f\in \F$). For any $p\ge 2$ and $u\ge 1$, it holds with probability at least $1-\exp(-u^2 (\gamma(\F,\psi_2)/d_{\psi_2}(\F))^2)$ that
\[
\sup_{f\in \F} \bigg|\frac{1}{N}\sum_{i=1}^N f^p(X_i) -\E f^p(X)\bigg|\lesssim_p u \frac{\gamma(\F,\psi_2)d^{p-1}_{\psi_2}(\F)}{\sqrt{N}}+u^p\frac{\gamma^p(\F,\psi_2)}{N}.
\]
As a corollary,
\[
\E \sup_{f\in \F} \bigg|\frac{1}{N}\sum_{i=1}^N f^p(X_i) -\E f^p(X)\bigg| \lesssim_p \frac{\gamma(\F,\psi_2)d^{p-1}_{\psi_2}(\F)}{\sqrt{N}}+\frac{\gamma^{p}(\F,\psi_2)}{N}.
\]
\end{theorem}

\begin{remark}
An alternative formulation of the tail bound states that, for any $u\ge 1$, it holds with probability at least $1-\exp(-u)$ that
\[
\sup _{f \in \F}\bigg|\frac{1}{N} \sum_{i=1}^N f^p\left(X_i\right)-\mathbb{E} f^p(X)\bigg| \lesssim_p  \frac{\gamma\left(\F, \psi_2\right)d^{p-1}_{\psi_2}(\F)}{\sqrt{N}}+ \frac{\gamma^p\left(\mathcal{F}, \psi_2\right)}{N}+ d^{p}_{\psi_2}(\F) \bigg(\sqrt{\frac{u}{N}}+ \frac{u^{p/2}}{N}\bigg).
\] 
The above high-probability bound is optimal. When $p=2$, it matches the tail bound obtained in \cite{dirksen2015tail,bednorz2014}, see also \cite[(1.18)]{mendelson2016upper} and \cite[Theorem 8]{koltchinskii2017concentration}. As noted in \cite{mendelson2016upper}, the optimality can be seen from setting $t=$ $u d_{\psi_2}^2(\F) / \sqrt{N}$ in Bernstein's inequality applied to $f^2$. When $p>2$, the optimality can be deduced from the $\alpha$-sub-exponential inequality \cite[Corollary 1.4]{gotze2021concentration}; see Lemma \ref{lemma:sub-exponential-1} for a formal statement. From our proof, the same high-probability and in-expectation bounds in Theorem \ref{thm:main2} also hold for the empirical process $
\sup_{f\in \F}\left|\frac{1}{N}\sum_{i=1}^N |f|^p(X_i)-\E|f|^p(X)\right|$, where $f^p(X_i)$ is replaced by its absolute value.
    
\end{remark}

\begin{remark}\label{remark:general_p}
For $p\ge 3/2$, a minor modification of our proof yields
\[
\E \sup_{f\in \F} \bigg|\frac{1}{N}\sum_{i=1}^N f^p(X_i) -\E f^p(X)\bigg| \lesssim_p \frac{\gamma(\F,\psi_2)d^{p-1}_{\psi_2}(\F)}{\sqrt{N}}+\frac{\gamma^{p}(\F,\psi_2)}{N}+\frac{\gamma^{3/2}(\F,\psi_2)d^{p-3/2}_{\psi_{2}}(\F)}{N^{3/4}}.
\]
If $p\ge 2$, the third term on the right-hand side is dominated by the first two terms. However, for $p\in [3/2,2)$, the third term remains significant and cannot be neglected. A natural question is to establish sharp concentration inequalities and expectation bounds for $
    \sup_{f\in \F}\left|\frac{1}{N}\sum_{i=1}^N |f|^p(X_i)-\E|f|^p(X)\right|$ for any $p>0$. We leave this question for future investigation.
\end{remark}

\section{Concentration of simple random tensors}\label{sec:concentration}

This section contains the proof of our first main result, Theorem \ref{thm:main1}. The upper bound for sub-Gaussian data follows from an application of Theorem \ref{thm:main2}, where the class $\F$ consists of all linear functionals in $H$ with norm one. To establish the matching lower bound for Gaussian data, we adapt techniques from \cite[Theorem 4]{koltchinskii2017concentration} to the higher-order setting; see Proposition \ref{prop:lowerbound} for the proof.

\begin{proof}[Proof of Theorem \ref{thm:main1}] By the expression \eqref{eq:compare_aux0} and Theorem \ref{thm:main2},
we have
\begin{align*}
 \E \bigg\|\frac{1}{N}\sum_{i=1}^{N} X_i^{\otimes p}-\E\, X^{\otimes p}\bigg\|&= \E \sup_{\|v\|= 1}\bigg|\frac{1}{N} \sum_{i=1}^N\left\langle X_i, v\right\rangle^p-\E\langle X, v\rangle^p\bigg|\\
   &=
\E \sup_{f\in \F} \bigg|\frac{1}{N}\sum_{i=1}^N f^p(X_i) -\E f^p(X)\bigg|\\ &\lesssim_p \frac{\gamma(\F,\psi_2)d^{p-1}_{\psi_2}(\F)}{\sqrt{N}}+\frac{\gamma^{p}(\F,\psi_2)}{N},
\end{align*}
where $\mathcal{F}:=\left\{\langle\cdot, v\rangle: v \in U_{H}\right\}, U_{H}:=\left\{v \in H:\|v\| =1\right\}$. 
It is clear that $\F$ is symmetric. Since $X$ is sub-Gaussian, the $\psi_2$-norm of linear functionals is equivalent to the $L_2$-norm. Hence,
\begin{align*}
d_{\psi_2}(\F)=\sup_{f\in\mathcal{F}}\|f\|_{\psi_2}\asymp \sup_{f\in\mathcal{F}}\|f\|_{L_2}=\sup_{\|v\|= 1}\left(\E\langle X,v\rangle^2\right)^{1/2}=\|\Sigma\|^{1/2}.
\end{align*}
Moreover, since $X$ is pre-Gaussian, there exists a centered Gaussian random variable $Y$ in $H$ with the same covariance $\Sigma$. This means that
\[
d_Y(u, v)
:= \left(\mathbb{E}(\langle Y, u\rangle-\langle Y, v\rangle)^2\right)^{1/2}=\langle u-v, \Sigma (u-v) \rangle^{1/2}
=\|\langle\cdot, u\rangle-\langle\cdot, v\rangle\|_{L_2(\mu)}, \quad u, v \in U_{H},
\]
where $\mu$ is the law of $X$. Using Talagrand's majorizing-measure theorem \cite[Theorem 2.10.1]{talagrand2022upper}, we deduce that
\[
\gamma\left(\mathcal{F}, \psi_2\right)\asymp \gamma\left(\mathcal{F}, L_2\right)=\gamma\left(U_{H}, d_Y\right) \asymp \mathbb{E} \sup _{u \in U_{H}}\langle Y, u\rangle = \E \|Y\| \asymp (\E \|Y\|^2)^{1/2}=\mathrm{Tr}(\Sigma)^{1/2}.
\]
Therefore, we obtain the in-expectation upper bound for sub-Gaussian data:
\begin{align*}
 \E \bigg\|\frac{1}{N}\sum_{i=1}^{N} X_i^{\otimes p}-\E\, X^{\otimes p}\bigg\| &\lesssim_p \frac{\gamma(\F,\psi_2)d^{p-1}_{\psi_2}(\F)}{\sqrt{N}}+\frac{\gamma^{p}(\F,\psi_2)}{N}
 \asymp \|\Sigma\|^{p/2} \bigg(\sqrt{\frac{r(\Sigma)}{N}}+\frac{r(\Sigma)^{p/2}}{N}\bigg).
\end{align*}
The high-probability upper bound follows analogously. The matching lower bound for Gaussian data is proved in Proposition \ref{prop:lowerbound} below.
\end{proof}

Now we prove the matching lower bound for Gaussian data in Theorem \ref{thm:main1}.

\begin{proposition}\label{prop:lowerbound}
Let $X, X_1, \ldots, X_N$ be i.i.d. centered Gaussian random variables in $H$ with covariance operator $\Sigma$. Then, for any integer $p\ge 2$,
\begin{align*}
\E \bigg\|\frac{1}{N}\sum_{i=1}^{N} X_i^{\otimes p}-\E\, X^{\otimes p}\bigg\| \gtrsim_p \|\Sigma\|^{p/2} 
\bigg(\sqrt{\frac{r(\Sigma)}{N}}+\frac{r(\Sigma)^{p/2}}{N}\bigg).
\end{align*}
\end{proposition}
\begin{proof}[Proof of Proposition \ref{prop:lowerbound}]

The proof builds upon ideas from \cite[Theorem 4]{koltchinskii2017concentration} and \cite[Lemma 5.6]{van2017structured}. Throughout the proof, $c(p)$, $c'(p)$ and $C(p)$ denote positive constants depending only on $p$. Recall that
\begin{align}\label{eq:lowerbound_aux0}
 \E \bigg\|\frac{1}{N}\sum_{i=1}^{N} X_i^{\otimes p}-\E\, X^{\otimes p}\bigg\|=\E \sup _{\|v\|=1}\bigg|\frac{1}{N} \sum_{i=1}^N\left\langle X_i, v\right\rangle^p-\E\langle X, v\rangle^p\bigg|.
\end{align}

We first prove that, for any $p\ge 2$,
\begin{align}\label{eq:lowerbound_aux1}
\E \sup_{\|v\|=1}\bigg|\frac{1}{N} \sum_{i=1}^N\left\langle X_i, v\right\rangle^p-\E\langle X, v\rangle^p\bigg| \gtrsim_p \|\Sigma\|^{p/2} \frac{r(\Sigma)^{p/2}}{N}.
\end{align}
Observe that
\begin{align}\label{eq:lowerbound_aux10}
   & \E \sup _{\|v\|=1}\bigg|\frac{1}{N} \sum_{i=1}^N\left\langle X_i, v\right\rangle^p-\E\langle X, v\rangle^p\bigg| \ge \E_{X_1} \sup _{\|v\|=1}\bigg|\E_{\{X_i\}_{i\ge 2}}\frac{1}{N} \sum_{i=1}^N\left(\left\langle X_i, v\right\rangle^p-\E\langle X, v\rangle^p\right)\bigg|\nonumber\\
    &= \frac{1}{N}\E_{X_1}\sup_{\|v\|=1} \bigg|\left\langle X_1, v\right\rangle^p-\E\langle X, v\rangle^p\bigg|\ge \frac{1}{N}\bigg(\E_{X_1}\Big[\,\sup_{\|v\|=1}\left|\langle X_1,v\rangle^p\right|\Big]-\sup_{\|v\|=1} \left|\E \langle X,v\rangle^p\right|\bigg) \nonumber\\
    &\overset{\text{($\star$)}}{\ge} \frac{1}{N} \left(\E \|X_1\|^{p}-C(p)\|\Sigma\|^{p/2}\right)\overset{\text{($\star\star$)}}
    \ge \frac{1}{N}\left( \mathrm{Tr}(\Sigma)^{p/2}- C(p)\|\Sigma\|^{p/2}\right)
    =\frac{1}{N}\|\Sigma\|^{p/2}\left(r(\Sigma)^{p/2}-C(p)\right).
    \end{align}
    Here, ($\star$) follows by $\E_{X_1}\left[ \sup_{\|v\|=1} |\langle X_1,v\rangle^p|\right]=\E \|X_1\|^p$ and
   $ \sup_{\|v\|=1} |\E \langle X,v\rangle^p|\le C(p) \|\Sigma\|^{p/2}$, since $\langle X,v \rangle\overset{\mathsf{d}}{\sim} \mathcal{N}(0,\langle \Sigma v,v\rangle)$, while ($\star\star$) follows by Jensen's inequality $\E \|X_1\|^p \ge \left(\E \|X_1\|^2\right)^{p/2}=\mathrm{Tr}(\Sigma)^{p/2}$ for $p\ge 2$. 
   
   We define the i.i.d. standard Gaussian random variables
\[
Z_i:=\frac{\langle X_i,v \rangle}{\langle \Sigma v ,v\rangle^{1/2}},\quad i=1,\ldots,N.
\] 
It follows that
\begin{align}\label{eq:lowerbound_aux3}
 \E \sup _{\|v\|=1}\bigg|\frac{1}{N} \sum_{i=1}^N\left\langle X_i, v\right\rangle^p-\E\langle X, v\rangle^p\bigg|
 &\ge \sup_{\|v\|=1}\E \bigg|\frac{1}{N}\sum_{i=1}^N \langle X_i,v \rangle^{p} -\E \langle X, v\rangle^{p} \bigg| \nonumber\\
 &=\sup_{\|v\|=1}\langle \Sigma v,v \rangle^{p/2} \E \bigg|\frac{1}{N}\sum_{i=1}^N Z_i^p-\E Z^p\bigg| \nonumber\\
 &\ge \|\Sigma\|^{p/2} \frac{c(p)}{\sqrt{N}}
 \ge \|\Sigma\|^{p/2}\frac{c(p)}{N}.
\end{align}
Therefore, combining \eqref{eq:lowerbound_aux10} and \eqref{eq:lowerbound_aux3} gives that, for any $t\in [0,1]$,
\begin{align*}
 \E \sup _{\|v\|=1}\bigg|\frac{1}{N} \sum_{i=1}^N\left\langle X_i, v\right\rangle^p-\E\langle X, v\rangle^p\bigg|
\ge \|\Sigma\|^{p/2}\left(t\frac{1}{N}\left(r(\Sigma)^{p/2}-C(p)\right)+(1-t)\frac{c(p)}{N}\right).
\end{align*}
Choosing $t$ to be small enough so that $tC(p)\le (1-t)c(p)$, we get
\[
\E \sup _{\|v\|=1}\bigg|\frac{1}{N} \sum_{i=1}^N\left\langle X_i, v\right\rangle^p-\E\langle X, v\rangle^p\bigg| \ge t\|\Sigma\|^{p/2} \frac{r(\Sigma)^{p/2}}{N}\asymp_p \|\Sigma\|^{p/2}\frac{r(\Sigma)^{p/2}}{N},
\]
as desired.

Next, we show that for any $p\ge 2$,
    \begin{align}\label{eq:lowerbound_aux2}
    \E \sup _{\|v\|=1}\bigg|\frac{1}{N} \sum_{i=1}^N\left\langle X_i, v\right\rangle^p-\E\langle X, v\rangle^p\bigg| \gtrsim_p \|\Sigma\|^{p/2} \sqrt{\frac{r(\Sigma)}{N}}.
\end{align}
Observe that
\begin{align}\label{eq:lowerbound2_aux1}
    \E \sup _{\|v\|=1}\bigg|\frac{1}{N} \sum_{i=1}^N\left\langle X_i, v\right\rangle^p-\E\langle X, v\rangle^p\bigg| &= \E \bigg\|\frac{1}{N}\sum_{i=1}^{N} X_i^{\otimes p}-\E\, X^{\otimes p}\bigg\|\nonumber\\
    &=\E \sup_{\|v\|=1}\bigg\|\frac{1}{N}\sum_{i=1}^N \langle X_i,v \rangle^{p-1} X_i -\E \langle X, v\rangle^{p-1}X \bigg\| \nonumber\\
    &\ge \sup_{\|v\|=1}\E \bigg\|\frac{1}{N}\sum_{i=1}^N \langle X_i,v \rangle^{p-1} X_i -\E \langle X, v\rangle^{p-1}X \bigg\|.
\end{align}
For a fixed $v$ with $\|v\| = 1$ and $\langle\Sigma v, v\rangle>0$, denote
\[
X^{\prime}:=X-\langle X, v\rangle \frac{\Sigma v}{\langle\Sigma v, v\rangle}.
\]
By a straightforward computation, for all $u \in H$, the random variables $\langle X, v\rangle$ and $\left\langle X^{\prime}, u\right\rangle$ are uncorrelated. Since they are jointly Gaussian, it follows that $\langle X, v\rangle$ and $X^{\prime}$ are independent. Define
\[
X_i^{\prime}:=X_i-\left\langle X_i, v\right\rangle \frac{\Sigma v}{\langle\Sigma v, v\rangle}, \quad i=1, \ldots, N.
\]
Then $\left\{X_i^{\prime}: i=1, \ldots, N\right\}$ and $\left\{\left\langle X_i, u\right\rangle: i=1, \ldots, N\right\}$ are also independent. We easily get
\begin{align}\label{eq:lowerbound2_aux2}
& \mathbb{E}\bigg\|\frac{1}{N} \sum_{i=1}^N\left\langle X_i, v\right\rangle^{p-1} X_i-\mathbb{E}\langle X, v\rangle^{p-1} X\bigg\| \nonumber\\
& \quad=\mathbb{E}\bigg\|\frac{1}{N} \sum_{i=1}^N\left(\left\langle X_i, v\right\rangle^p-\mathbb{E}\langle X, v\rangle^p\right) \frac{\Sigma v}{\langle\Sigma v, v\rangle}+\frac{1}{N} \sum_{i=1}^N\left\langle X_i, v\right\rangle^{p-1} X_i^{\prime}\bigg\|,
\end{align}
where we used the fact that
\[
\mathbb{E}\langle X, v\rangle^{p-1} X=\mathbb{E}\langle X, v\rangle^p \frac{\Sigma v}{\langle\Sigma v, v\rangle}+\mathbb{E}\langle X, v\rangle^{p-1} \mathbb{E} X^{\prime}=\mathbb{E}\langle X, v\rangle^p \frac{\Sigma v}{\langle\Sigma v, v\rangle}.
\]

Note that, conditionally on $\left\langle X_i, v\right\rangle, i=1, \ldots, N$, the distribution of the random variable
\[
\frac{1}{N} \sum_{i=1}^{N}\left\langle X_i, v\right\rangle^{p-1} X_i^{\prime}
\]
is Gaussian and it coincides with the distribution of the random variable
\[
\bigg(\frac{1}{N} \sum_{i=1}^{N}\left\langle X_i, v\right\rangle^{2(p-1)}\bigg)^{1 / 2} \frac{X^{\prime}}{\sqrt{N}}.
\]
Denote now by $\mathbb{E}_v$ the conditional expectation given $\left\langle X_i, v\right\rangle, i=1, \ldots, N$ and by $\mathbb{E}^{\prime}$ the conditional expectation given $X_1^{\prime}, \ldots, X_N^{\prime}$. Then, we have
\begin{align}\label{eq:lowerbound2_aux3}
& \mathbb{E}\bigg\|\frac{1}{N} \sum_{i=1}^{N}\left(\left\langle X_i, v\right\rangle^{p}-\mathbb{E}\langle X, v\rangle^{p}\right) \frac{\Sigma v}{\langle\Sigma v, v\rangle}+\frac{1}{N} \sum_{i=1}^{N}\left\langle X_i, v\right\rangle^{p-1} X_i^{\prime}\bigg\| \nonumber\\
& \quad=\E \mathbb{E}_{v}\bigg\|\frac{1}{N} \sum_{i=1}^{N}\left(\left\langle X_i, v\right\rangle^{p}-\mathbb{E}\langle X, v\rangle^p\right) \frac{\Sigma v}{\langle\Sigma v, v\rangle}+\frac{1}{N} \sum_{i=1}^{N}\left\langle X_i, v\right\rangle^{p-1} X_i^{\prime}\bigg\| \nonumber \\
& \quad=\mathbb{E} \mathbb{E}_{v}\bigg\|\frac{1}{N} \sum_{i=1}^{N}\left(\left\langle X_i, v\right\rangle^{p}-\mathbb{E}\langle X, v\rangle^p\right) \frac{\Sigma v}{\langle\Sigma v, v\rangle}+\bigg(\frac{1}{N} \sum_{i=1}^{N}\left\langle X_i, v\right\rangle^{2(p-1)}\bigg)^{1 / 2} \frac{X^{\prime}}{\sqrt{N}}\bigg\| \nonumber\\
& \quad=\mathbb{E}\bigg\|\frac{1}{N} \sum_{i=1}^{N}\left(\left\langle X_i, v\right\rangle^{p}-\mathbb{E}\langle X, v\rangle^p\right) \frac{\Sigma v}{\langle\Sigma v, v\rangle}+\bigg(\frac{1}{N} \sum_{i=1}^{N}\left\langle X_i, v\right\rangle^{2(p-1)}\bigg)^{1 / 2} \frac{X^{\prime}}{\sqrt{N}}\bigg\|.
\end{align}
Also
\begin{align*}
    &\mathbb{E}\bigg\|\frac{1}{N} \sum_{i=1}^{N}\left(\left\langle X_i, v\right\rangle^{p}-\mathbb{E}\langle X, v\rangle^p\right) \frac{\Sigma v}{\langle\Sigma v, v\rangle}+\bigg(\frac{1}{N} \sum_{i=1}^{N}\left\langle X_i, v\right\rangle^{2(p-1)}\bigg)^{1 / 2} \frac{X^{\prime}}{\sqrt{N}}\bigg\|\\
    &\quad= \E \E^{\prime}\bigg\|\frac{1}{N} \sum_{i=1}^{N}\left(\left\langle X_i, v\right\rangle^{p}-\mathbb{E}\langle X, v\rangle^p\right) \frac{\Sigma v}{\langle\Sigma v, v\rangle}+\bigg(\frac{1}{N} \sum_{i=1}^{N}\left\langle X_i, v\right\rangle^{2(p-1)}\bigg)^{1 / 2} \frac{X^{\prime}}{\sqrt{N}}\bigg\|\\
    &\quad\ge \E \bigg\|\E^{\prime}\frac{1}{N} \sum_{i=1}^{N}\left(\left\langle X_i, v\right\rangle^{p}-\mathbb{E}\langle X, v\rangle^p\right) \frac{\Sigma v}{\langle\Sigma v, v\rangle}+\E^{\prime}\bigg(\frac{1}{N} \sum_{i=1}^{N}\left\langle X_i, v\right\rangle^{2(p-1)}\bigg)^{1 / 2} \frac{X^{\prime}}{\sqrt{N}}\bigg\|\\
    &\quad = \E \bigg(\frac{1}{N} \sum_{i=1}^{N}\left\langle X_i, v\right\rangle^{2(p-1)}\bigg)^{1 / 2} \frac{\E\|X^{\prime}\|}{\sqrt{N}}.
\end{align*}
Note that
\[
\mathbb{E}|\langle X, v\rangle|=\sqrt{\frac{2}{\pi}}\langle\Sigma v, v\rangle^{1 / 2}.
\]
Therefore,
\[
\E\|X^{\prime}\| \geq \mathbb{E}\|X\|-\mathbb{E}|\langle X, v\rangle| \frac{\|\Sigma v\|}{\langle\Sigma v, v\rangle}=\mathbb{E}\|X\|-\sqrt{\frac{2}{\pi}} \frac{\|\Sigma v\|}{\langle\Sigma v, v\rangle^{1 / 2}}
\]
and
\begin{align*}
&\mathbb{E}\bigg\|\frac{1}{N} \sum_{i=1}^{N}\left(\left\langle X_i, v\right\rangle^{p}-\mathbb{E}\langle X, v\rangle^p\right) \frac{\Sigma v}{\langle\Sigma v, v\rangle}+\bigg(\frac{1}{N} \sum_{i=1}^{N}\left\langle X_i, v\right\rangle^{2(p-1)}\bigg)^{1 / 2} \frac{X^{\prime}}{\sqrt{N}}\bigg\|\\
&\quad \ge \E \bigg(\frac{1}{N} \sum_{i=1}^{N}\left\langle X_i, v\right\rangle^{2(p-1)}\bigg)^{1 / 2} \frac{\mathbb{E}\|X\|-\sqrt{2/\pi} \|\Sigma v\|/\langle\Sigma v, v\rangle^{1 / 2}}{\sqrt{N}}\\
&\quad =\langle \Sigma v ,v\rangle^{(p-1)/2} \E \bigg(\frac{1}{N} \sum_{i=1}^{N}Z_i^{2(p-1)}\bigg)^{1 / 2} \frac{\mathbb{E}\|X\|-\sqrt{2/\pi} \|\Sigma v\|/\langle\Sigma v, v\rangle^{1 / 2}}{\sqrt{N}},
\end{align*}
where we recall that
\[
Z_i=\frac{\langle X_i,v \rangle}{\langle \Sigma v ,v\rangle^{1/2}},\quad i=1,\ldots,N
\]
are i.i.d. standard Gaussian random variables. One can check that
\[
\E \bigg(\frac{1}{N} \sum_{i=1}^{N}Z_i^{2(p-1)}\bigg)^{1 / 2} \ge c^{\prime}(p)
\]
for a positive constant $c^{\prime}(p)$ depending only on $p$, which implies that
\begin{align*}
&\mathbb{E}\bigg\|\frac{1}{N} \sum_{i=1}^{N}\left(\left\langle X_i, v\right\rangle^{p}-\mathbb{E}\langle X, v\rangle^p\right) \frac{\Sigma v}{\langle\Sigma v, v\rangle}+\bigg(\frac{1}{N} \sum_{i=1}^{N}\left\langle X_i, v\right\rangle^{2(p-1)}\bigg)^{1 / 2} \frac{X^{\prime}}{\sqrt{N}}\bigg\|\\
&\quad \ge c^{\prime}(p)\langle\Sigma v ,v\rangle^{(p-1)/2} \frac{\mathbb{E}\|X\|-\sqrt{2/\pi} \|\Sigma v\|/\langle\Sigma v, v\rangle^{1 / 2}}{\sqrt{N}}\\
&\quad =\frac{c^{\prime}(p)}{\sqrt{N}}\bigg(\langle\Sigma v ,v\rangle^{(p-1)/2}\E \|X\|-\sqrt{\frac{2}{\pi}}\|\Sigma v\|\langle \Sigma v, v\rangle^{(p-2)/2}\bigg).
\end{align*}

We combine this bound with \eqref{eq:lowerbound2_aux1}, \eqref{eq:lowerbound2_aux2} and \eqref{eq:lowerbound2_aux3} to get
\begin{align}\label{eq:lowerbound_aux7}
 \E \sup _{\|v\|=1}\bigg|\frac{1}{N} \sum_{i=1}^N\left\langle X_i, v\right\rangle^p-\E\langle X, v\rangle^p\bigg| &\ge \frac{c^{\prime}(p)}{\sqrt{N}}\sup_{\|v\|=1}\bigg(\langle\Sigma v ,v\rangle^{(p-1)/2}\E \|X\|-\sqrt{\frac{2}{\pi}}\|\Sigma v\|\langle \Sigma v, v\rangle^{(p-2)/2}\bigg)\nonumber\\
 &\overset{\text{($\star$)}}{\ge} \frac{c^{\prime}(p)}{\sqrt{N}} \bigg(\|\Sigma\|^{(p-1)/2}\E \|X\|-\sqrt{\frac{2}{\pi}}\|\Sigma\|^{p/2}\bigg)\nonumber\\
 &\overset{\text{($\star\star$)}}{\ge} c^{\prime}(p)\|\Sigma\|^{p/2}\bigg(\frac{c^{\prime}\sqrt{r(\Sigma)}-\sqrt{2/\pi}}{\sqrt{N}}\bigg),
\end{align}
where ($\star$) follows by taking the eigenvector of $\Sigma$ corresponding to the largest eigenvalue, and ($\star\star$) follows since for a Gaussian vector $X$, $\E \|X\|\asymp \sqrt{\E \|X\|^2}=\sqrt{\mathrm{Tr}(\Sigma)}$.

Recall that we have the following bound by \eqref{eq:lowerbound_aux3}
\begin{align}\label{eq:lowerbound_aux8}
  \E \sup _{\|v\|=1}\bigg|\frac{1}{N} \sum_{i=1}^N\left\langle X_i, v\right\rangle^p-\E\langle X, v\rangle^p\bigg| \ge \|\Sigma\|^{p/2} \frac{c(p)}{\sqrt{N}}.
\end{align}
Combining \eqref{eq:lowerbound_aux7} and \eqref{eq:lowerbound_aux8} gives that,  for any $t\in [0,1]$,
\begin{align*}
  \E \sup _{\|v\|=1}\bigg|\frac{1}{N} \sum_{i=1}^N\left\langle X_i, v\right\rangle^p-\E\langle X, v\rangle^p\bigg| 
\ge \|\Sigma\|^{p/2} \bigg(tc^{\prime}(p)\bigg(\frac{c^{\prime}\sqrt{r(\Sigma)}-\sqrt{2/\pi}}{\sqrt{N}}\bigg)+(1-t)\frac{c(p)}{\sqrt{N}}\bigg).
\end{align*}
Choosing $t$ to be small enough so that $tc^{\prime}(p)\sqrt{2/\pi}\le (1-t)c(p)$, we get
\[
 \E \sup _{\|v\|=1}\bigg|\frac{1}{N} \sum_{i=1}^N\left\langle X_i, v\right\rangle^p-\E\langle X, v\rangle^p\bigg|\ge t c^{\prime}(p)c^{\prime}\|\Sigma\|^{p/2} \sqrt{\frac{r(\Sigma)}{N}}\asymp_p \|\Sigma\|^{p/2} \sqrt{\frac{r(\Sigma)}{N}}.
\]
This completes the proof of \eqref{eq:lowerbound_aux2}. Combining \eqref{eq:lowerbound_aux0}, \eqref{eq:lowerbound_aux1} and \eqref{eq:lowerbound_aux2} completes the proof.
\end{proof}

\section{$L_p$ empirical processes }\label{sec:multi}

This section studies $L_p$ empirical processes
\begin{align}\label{eq:higher-order-processes}
f \mapsto 
    \frac{1}{N} \sum_{i=1}^N f^p(X_i) 
    - \E f^p(X),
    \qquad f \in \mcF.
\end{align}
We leverage generic chaining techniques to derive sharp high-probability upper bounds on their suprema in terms of quantities that reflect the complexity of the function class $\mcF$.

Our theory builds on the framework introduced by Mendelson in \cite{mendelson2016upper}. The central idea is as follows. First, we gather precise structural information on the typical coordinate projection of $\F$; that is, for $\sigma=(X_1,\ldots,X_N)$, we study the structure of
\[
P_{\sigma}(\F)=\left\{(f(X_i))_{i=1}^{N}: f\in \F\right\}.
\]
Second, for a typical $\sigma=(X_1,\ldots,X_N)$, we analyze the suprema of the conditioned Bernoulli processes:
\[
f \mapsto \sum_{i=1}^N \varepsilon_i f^p(X_i)  \  \Big | \sigma,
\]
where $\varepsilon_1,\ldots,\varepsilon_N$ are independent symmetric Bernoulli random variables. 

The structure of typical coordinate projections of sub-Gaussian function classes has been extensively studied in empirical process theory \cite{mendelson2010empirical, mendelson2011discrepancy,mendelson2012generic,mendelson2016dvoretzky,mendelson2016upper}, primarily in the context of quadratic and product empirical processes. Here, we extend the analysis to empirical processes of arbitrary higher order. A key component of our approach is Theorem \ref{thm:l_m_norm}, which establishes a sharp uniform bound on the $\ell_{m}$-norm of the coordinate projection vector of the function class $\F$ for any $m\ge 1$. Specifically, it controls
\[
\sup_{f\in \F} \bigg(\sum_{i=1}^{N}|f|^{m}\left(X_i\right)\bigg)^{1/m}
\]
in terms of quantities that capture the geometric complexity of $\F$.

The rest of this section is organized as follows. Subsection \ref{subsec:background} introduces some basic definitions from Mendelson's work \cite{mendelson2016upper} and discusses an important result, Lemma \ref{lemma:mendelson} (\cite[Corollary 3.6]{mendelson2016upper}), which plays a crucial role in our proof of Theorem \ref{thm:main2}. We then prove Theorem \ref{thm:main2} in Subsection \ref{subsec:main2} and prove Theorem \ref{thm:l_m_norm} in Subsection \ref{subsec:l_m_norm}.

\subsection{Background and auxiliary results}\label{subsec:background}
Following \cite{mendelson2016upper}, we will utilize complexity parameters that take into account all the $L_p$ structures endowed by the process. For $q \geq 1$, we recall that the graded $L_q$ norm is defined by
\[
\|f\|_{(q)} :=\sup _{1 \leq p \leq q} \frac{\|f\|_{L_p}}{\sqrt{p}}.
\]
We will use repeatedly that, for any $q \ge 2,$ $\|f \|_{L_2} \lesssim \|f\|_{(q)} \lesssim \|f\|_{\psi_2}.$
As shown in \cite{mendelson2016upper}, the graded $\gamma$-type functionals that we now define serve as key parameters to characterize the complexity of the typical coordinate projection of $\F$.

\begin{definition}[Graded $\gamma$-type functionals]
Given a class of functions $\F$ on a probability space $(\Omega, \mu)$, constants $u \geq 1$ and $s_0 \geq 0$, set
\[
\Lambda_{s_0, u}(\F) :=\inf \sup _{f \in \F} \sum_{s \geq s_0} 2^{s / 2}\left\|f-\pi_s f\right\|_{(u^2 2^s)},
\]
where the infimum is taken over all admissible sequences, and $\pi_s f$ is the nearest point in $\F_s$ to $f$ with respect to the $(u^2 2^s)$-norm. Also set
\[
\widetilde{\Lambda}_{s_0, u}(\F) :=\Lambda_{s_0, u}(\F)+2^{s_0 / 2} \sup_{f \in \F}\left\|\pi_{s_0} f\right\|_{(u^2 2^{s_0})} .
\]
\end{definition} 

 The following lemma shows that $\Lambda_{s_0, u}(\F)$ is upper bounded by $\gamma(\F,\psi_2)$ up to an absolute constant.
\begin{lemma}\label{lemma:Lambda_Gamma}
    For any $u\ge 1$ and $s_0\ge 0$, $\Lambda_{s_0,u}(\F)\lesssim \gamma(\F,\psi_2)$.
\end{lemma}
\begin{proof}
For any admissible sequence $(\F_s)_{s\ge 0} \subset \F$, any $f\in \F$ and $s\ge 0$, let $\bar{\pi}_s f$ be the nearest point in $\F_s$ to $f$ with respect to the $\psi_2$ norm. By the definition of $\pi_s f$, we have
\[
\sum_{s \geq s_0} 2^{s / 2}\left\|f-\pi_s f\right\|_{(u^2 2^s)}\le \sum_{s \geq s_0} 2^{s / 2}\left\|f-\bar{\pi}_s f\right\|_{(u^2 2^s)}\lesssim \sum_{s \geq s_0} 2^{s / 2}\left\|f-\bar{\pi}_s f\right\|_{\psi_2}.
\]
Thus,
\[
\Lambda_{s_0, u}(\F)
\lesssim \inf \sup _{f \in \F} \sum_{s \geq s_0} 2^{s / 2} \left\|f-\bar{\pi}_s f\right\|_{\psi_2} 
\le \inf \sup _{f \in \F} \sum_{s \geq 0} 2^{s / 2} \left\|f-\bar{\pi}_s f\right\|_{\psi_2} = \gamma(\F,\psi_2). \qedhere
\]
\end{proof}

To analyze (conditioned) Bernoulli processes, we will use a chaining argument combined with the following H\"offding’s inequality.

\begin{lemma}[H\"offding's Inequality, \cite{montgomery1990distribution}]\label{lemma:Hoffding} Let $\varepsilon_1,\ldots,\varepsilon_N$ be independent symmetric Bernoulli random variables (that is, $\P(\varepsilon_n=+1)=\P(\varepsilon_n=-1)=1/2$). For any fixed $I \subset\{1, \ldots, N\}$, $z \in \mathbb{R}^N$, it holds with probability at least $1-2 \exp (-t^2 / 2 )$ that
\[
\bigg|\sum_{i=1}^N \varepsilon_i z_i\bigg| \le \sum_{i \in I}\left|z_i\right|+t\bigg(\sum_{i \in I^c} z_i^2\bigg)^{1 / 2}.
\]
In particular, for every $1 \le k \le N,$ it holds with probability at least $1-2 \exp (-t^2 / 2 )$ that
\[
\bigg|\sum_{i=1}^N \varepsilon_i z_i\bigg| 
\le \sum_{i=1}^k \left|z_i^*\right|+t\bigg(\sum_{i=k+1}^N (z_i^*)^2\bigg)^{1 / 2},
\]
where $\left(z_i^*\right)_{i=1}^N$ denotes the non-increasing rearrangement of $\left(\left|z_i\right|\right)_{i=1}^N.$ This estimate is optimal when $k\asymp t^2$.
\end{lemma}

According to Lemma \ref{lemma:Hoffding}, the influence of $z$ on the Bernoulli process depends on the $\ell_1$-norm of the largest $k$ coordinates of $(|z_i|)_{i=1}^{N}$ and the $\ell_2$-norm of the smallest $N-k$ coordinates. As noted in \cite{mendelson2016upper}, the graded $L_q$ norm plays a key role in analyzing the monotone non-increasing rearrangement of $N$ independent copies of $Z$. In particular, \cite[Corollary 3.6]{mendelson2016upper} provides a useful way to decompose a collection of random vectors into their largest and smallest coordinate components, a technique central to \cite{mendelson2016upper}. Below, we present the statement obtained by setting $q=2r$ and $\beta=1/2$ in their notation.

\begin{lemma}[{\cite[Corollary 3.6]{mendelson2016upper}}]\label{lemma:mendelson}
There exist absolute constants $c_0, c_1$, and for every $r \ge 1$, there exist constants $c_2, c_3, c_4$ and $c_5$ depending only on $r$, such that the following holds. Set
\begin{align}\label{eq:j_s}
j_s =\min \left\{\left\lceil
\frac{c_0 u^2 2^s}{\log \left(4+e N / u^2 2^s\right)}\right\rceil, N+1\right\}
\end{align}
for $u \geq c_2$. Let $\mathcal{H} \subset L_{2r}$  be a function class with cardinality at most $2^{2^{s+2}}$. Then, with probability at least $1-2 \exp \left(-c_3 u^2 2^s\right)$, the following holds: for every $h\in \mathcal{H}$, the random vector $\left(h\left(X_i\right)\right)_{i=1}^N$ can be decomposed as $ \left(h\left(X_i\right)\right)_{i=1}^N = U_h+V_h$, where $U_h\in \R^N$ is supported on the largest $j_s-1$ coordinates of $\left(\left|h\left(X_i\right)\right|\right)_{i=1}^N$ while $V_h\in \R^{N}$ is supported on the complement, and
\begin{align*}
\left\|U_h\right\|_{\ell_2} \leq c_4 u 2^{s / 2}\|h\|_{(c_5 u^2 2^s)},
\qquad\left\|V_h\right\|_{\ell_r} \leq c_1 N^{1 / r}\|h\|_{L_{2r}}.
\end{align*}

\end{lemma}

Let $(\F_s)_{s\ge 0}\subset \F$ be an admissible sequence. For $f\in \F$ and $s\ge 0$, we define $\Delta_s f := \pi_{s+1}f-\pi_sf$. To analyze $L_p$ empirical processes, we apply Lemma \ref{lemma:mendelson} to the classes $\F_s$ and $\{\Delta_s f:f\in \F\}$ with carefully chosen values of $r$ that depend on $p$, leading to the following corollary.

\begin{corollary}\label{coro:1}
 Let $\left(\F_s\right)_{s \geq 0}$ be an admissible sequence of $\F$, and let $(j_s)_{s\ge s_0}$ be the sequence defined in \eqref{eq:j_s}. Let $s_0$ be a fixed index and $p \ge 1$.  Denote $\Delta_s f := \pi_{s+1}f-\pi_sf.$ There are constants $c_1, c_2$ and $c_3$ that depend only on $p$ and for $u\ge c(p)$ where $c(p)$ is a constant depending only on $p$, an event of probability at least $1-2 \exp \left(-c_1 u^2 2^{s_0}\right)$ on which the following holds. For every $f \in \F$ and $s\ge s_0$,
\begin{align*}
\bigg(\sum_{i<j_s}\left[\left(\Delta_s f\right)^2\left(X_i\right)\right]^*\bigg)^{1/2} 
&\le c_2 u 2^{s / 2}\left\|\Delta_s f\right\|_{(c_3 u^2 2^s)},\\ 
\bigg(\sum_{i \geq j_s}\left[\left(\Delta_s f\right)^{4}\left(X_i\right)\right]^*\bigg)^{1/4} 
&\le c_2 N^{1 / 4}\left\|\Delta_s f\right\|_{L_{8}},\\
\bigg(\sum_{i<j_s}\left[\left(\pi_s f\right)^2\left(X_i\right)\right]^*\bigg)^{1/2} 
&\le c_2 u 2^{s / 2}\left\|\pi_s f\right\|_{(c_3 u^2 2^s)},\\
\bigg(\sum_{i \geq j_s}\left[\left(\pi_s f\right)^{4(p-1)}\left(X_i\right)\right]^*\bigg)^{\frac{1}{4(p-1)}} &\le c_2  N^{\frac{1}{4(p-1)}}\left\|\pi_s f\right\|_{L_{8(p-1)}}.
\end{align*}
\end{corollary}
\begin{proof}[Proof of Corollary \ref{coro:1}]
    By definition of an admissible sequence, we have $|\F_s| \le 2^{2^s}.$ Further, $|\F_s|\cdot |\F_{s+1}| \le 2^{2^s} \cdot 2^{2^{s+1}}\le 2^{2^{s+2}},$ which implies that for $\Delta_s f = \pi_{s+1}f-\pi_sf$, $|\{\Delta_s f: f \in \F\}| \le2^{2^{s+2}}$. Therefore, the result follows by applying Lemma~\ref{lemma:mendelson}, for every $s\ge s_0$, to the class $\F_s$ with the choice $r=4(p-1)\ge 1$ and to the class $\{\Delta_s f:f\in \F\}$ with $r=4$, and taking a union bound to get the probability estimate
    \[
1-2\sum_{s\ge s_0} 2\exp(-cu^2 2^{s})\cdot 2^{2^{s+2}}\ge 1-2\exp(-c_1u^2 2^{s_0}),
    \]
    where $c_1$ is a constant depending only on $p$.
\end{proof}

\subsection{Proof of Theorem \ref{thm:main2}}\label{subsec:main2}
This subsection contains the proof of our second main result, Theorem \ref{thm:main2}. The proof relies on the next result, Theorem \ref{thm:l_m_norm}, which establishes a uniform bound on the $\ell_{m}$-norm of the coordinate projection vector of the function class $\F$ and improves upon the estimates in \cite[(3.8)]{mendelson2010empirical}. 

\begin{theorem}\label{thm:l_m_norm} 
Let $X_1, \ldots, X_N \iid \mu$ be a sequence of random variables and let $\F$ be a function class. For any $m\ge 1$, 
    \begin{align*}
    \sup_{f\in \F} \bigg(\sum_{i=1}^{N}|f|^{m}\left(X_i\right)\bigg)^{1/m}&\lesssim_{m} \gamma(\F,\psi_2)+N^{1/m}d_{\psi_2}(\F)+N^{1/2m}d^{(m-1)/m}_{\psi_2}(\F)\gamma^{1/m}(\F,\psi_2)\\
    &\quad +N^{1/2m}d_{\psi_2}(\F) (\phi^{-1}(Z))^{1/m},
    \end{align*}
    where \[
    \phi(x)=2^{\min\{(\sqrt{N}x)^{2/m},\, x^2\}}-1,\quad x\ge 0,
    \] 
    and $\phi^{-1}(x)$ denotes its inverse function. Here, $Z$ is a nonnegative random variable satisfying $\E \, Z\le C$ for some absolute constant $C>0$. Moreover, if $m\ge 2$, then
    \[
    \sup_{f\in \F} \bigg(\sum_{i=1}^{N}|f|^{m}\left(X_i\right)\bigg)^{1/m}\lesssim_{m} \gamma(\F,\psi_2)+N^{1/m}d_{\psi_2}(\F)+N^{1/2m}d_{\psi_2}(\F) (\phi^{-1}(Z))^{1/m}.
    \]
\end{theorem}
\begin{remark}\label{rem:lm_norm}
    An application of Markov's inequality to Theorem~\ref{thm:l_m_norm} yields the following high-probability bound: for any $m\ge 2$ and $w>0$, it holds with probability at least $1-2\exp(-\min\{ (\sqrt{N}w)^{2/m},w^2\})$ that
    \[
     \sup_{f\in \F} \bigg(\sum_{i=1}^{N}|f|^{m}\left(X_i\right)\bigg)^{1/m}\lesssim_{m} \gamma(\F,\psi_2)+N^{1/m}d_{\psi_2}(\F)+N^{1/2m}d_{\psi_2}(\F) w^{1/m}.
    \]
An alternative formulation of the tail bound states that for any $u\ge 1$, with probability at least $1-2\exp(-u^2)$,
\begin{align}\label{eq:rem_lm_aux1}
\sup _{f \in \mathcal{F}}\bigg(\sum_{i=1}^N|f|^m\left(X_i\right)\bigg)^{1 / m} &\lesssim_m \gamma\left(\mathcal{F}, \psi_2\right)+N^{1 / m} d_{\psi_2}(\mathcal{F})+N^{1 / 2 m} d_{\psi_2}(\mathcal{F}) \left(u^{1/m}+uN^{-1/2m}\right) \nonumber\\
&=\gamma\left(\mathcal{F}, \psi_2\right)+N^{1 / m} d_{\psi_2}(\mathcal{F})+d_{\psi_2}(\F)u+N^{1/2m}d_{\psi_2}(\F)u^{1/m} \nonumber\\
&\asymp \gamma\left(\mathcal{F}, \psi_2\right)+N^{1 / m} d_{\psi_2}(\mathcal{F})+d_{\psi_2}(\F)u,
\end{align}
where the last step follows from the inequality $N^{1/2m}u^{1/m}\le \frac{1}{2}(N^{1/m}+u^{2/m})\le N^{1/m}+u$, for $u\ge 1$ and $m\ge 2$. 

We demonstrate the optimality of the bound \eqref{eq:rem_lm_aux1} via a simple example. Let $H=\R^d$ and consider a symmetric compact set $K\subset \R^d$. Define the function class $\F=\left\{\langle \cdot,v \rangle: v\in K\right\}$, which consists of linear functions indexed by $K$. Let $X_1,\ldots,X_N$ be i.i.d. standard Gaussian random vectors in $\R^d$, and let $v^{*}\in K$ be a point attaining the maximal Euclidean norm on $\R^d$, i.e., $ \|v^{*}\|_{\ell_2}=\sup_{v\in K}\|v\|_{\ell_2}$. Observe that
\begin{align*}
\E \sup _{f \in \mathcal{F}}\bigg(\sum_{i=1}^N|f|^m\left(X_i\right)\bigg)^{1 / m} = \E \sup_{v\in K} \bigg(\sum_{i=1}^{N} |\langle X_i,v \rangle|^m \bigg)^{1/m}&\overset{\hspace{-0.1cm}\text{($\star$)}}{\gtrsim_m} \E \sup_{v\in K} |\langle X_1,v\rangle| +  N^{1/m}\|v^{*}\|_{\ell_2}\\
&\overset{\text{($\star\star$)}}{\asymp} \gamma(\F,\psi_2)+N^{1/m}d_{\psi_2}(\F) ,
\end{align*}
where ($\star$) follows by separately analyzing the case $N=1$ and the contribution of the maximal element $v^{*}\in K$, and  ($\star\star$) follows by Talagrand's majorizing-measure theorem \cite[Theorem 2.10.1]{talagrand2022upper}, which gives $\E\sup_{v\in K}\langle X_1,v \rangle\asymp \gamma(\F,\psi_2) $, along with the relation $\|v^{*}\|_{\ell_2}=\sup_{v\in K}\|v\|_{\ell_2}\asymp \sup_{f\in \F}\|f\|_{\psi_2}= d_{\psi_2}(\F)$. This confirms the necessity of the deterministic terms in \eqref{eq:rem_lm_aux1}. Furthermore, note that
\[
\sup_{v\in K} \bigg(\sum_{i=1}^{N} |\langle X_i,v \rangle|^m \bigg)^{1/m}\ge \bigg(\sum_{i=1}^{N} |\langle X_i,v^{*} \rangle|^m \bigg)^{1/m}\ge |\langle X_1,v^{*} \rangle|\overset{\mathsf{d}}{\sim} \|v^{*}\|_{2} \cdot |y|,
\]
where $y$ is a standard Gaussian random variable. This confirms the sub-Gaussian tail behavior in \eqref{eq:rem_lm_aux1}.
\end{remark}

We defer the proof of Theorem \ref{thm:l_m_norm} to Subsection \ref{subsec:l_m_norm} and proceed now to prove Theorem \ref{thm:main2}.

\begin{proof}[Proof of Theorem \ref{thm:main2}]

We first observe that the condition $0 \in \mcF$ or the symmetry of $\mcF$ (i.e., $f\in \mcF$ implies $-f \in \mcF$) ensures that $\gamma(\mcF, \psi_2)\gtrsim d_{\psi_2}(\mcF)$ (Lemma \ref{lemma:gamma_2_d_psi_2}), a fact that we will use repeatedly in the proof.

Let $(\mcF_s)_{s\ge 0}$ be an admissible sequence of $\mcF.$ Let $u\ge 4$ and $s_0\ge 0$. For any $f\in \F$ and $s\ge 0$, let $\pi_s f$ denote the nearest point in $\F_s$ to $f$ with respect to the $(u^2 2^s)$-norm. 

The proof is based on a symmetrization argument, an important tool in empirical process theory \cite[Chapter 2.3]{van2023weak}. Let $(\varepsilon_i)_{i=1}^{N}$ be independent symmetric Bernoulli random variables that are independent of $(X_i)_{i=1}^N$. We seek to establish a high-probability upper bound on the supremum of the Bernoulli process given by
\[
\sup_{f\in \F} \bigg|\sum_{i=1}^N \varepsilon_i f^p(X_i)\bigg|.
\]
To that end, we first note that any $f^p(X_i)$ can be represented as a chain, or telescoping sum, of links of the form $(\pi_{s+1}f)^p (X_i)-(\pi_{s}f)^p(X_i).$ Namely 
\[
f^p(X_i)=\sum_{s\ge s_0} ((\pi_{s+1}f)^p -(\pi_{s}f)^p) (X_i)+(\pi_{s_0}f)^p(X_i),
\]
provided that $\pi_{s} f\to f$ as $s \to \infty$.
By the triangle inequality, we have that
\begin{align}\label{eq:main2_aux1}
\bigg|\sum_{i=1}^N \varepsilon_i f^p(X_i)\bigg|\le \sum_{s\ge s_0} \bigg|\sum_{i=1}^N \varepsilon_i ((\pi_{s+1}f)^p-(\pi_s f)^p)(X_i)\bigg|+\bigg|\sum_{i=1}^N \varepsilon_i (\pi_{s_0} f)^p(X_i) \bigg|.
\end{align}

To ease notation, we write $\Delta_s f:=\pi_{s+1} f-\pi_s f$ and $(\Delta_s f)_i := (\Delta_s f)(X_i)$. We use similar notation throughout the remainder of the proof. For each $s\ge s_0$, we apply Lemma \ref{lemma:Hoffding} with $t=u 2^{s/2}$ and a set $I_s \subset \{1,\dots, N\}$ (to be specified later) to the summand indexed by $s$ in \eqref{eq:main2_aux1}. This implies that, with probability at least $1-2 \exp (-u^2 2^{s-1})$, the following holds:
\begin{align*}
\bigg|\sum_{i=1}^N \varepsilon_i ((\pi_{s+1}f)^p-(\pi_s f)^p)_i\bigg| 
&\le \sum_{i\in I_s} |((\pi_{s+1}f)^p-(\pi_s f)^p)_i|
+ u2^{s/2}\bigg(\sum_{i\in I_s^c}((\pi_{s+1}f)^p-(\pi_s f)^p)_i^2\bigg)^{1/2}\\
&\overset{\text{($\star$)}}{\le} 
p\sum_{i\in I_s} |(\Delta_s f)_i|\left(|\pi_{s+1}f|^{p-1}+|\pi_{s}f|^{p-1}\right)_i\\
&\quad + u 2^{s/2} p \bigg(\sum_{i\in I_s^c} (\Delta_s f)_i^2\left(|\pi_{s+1}f|^{p-1}+|\pi_{s}f|^{p-1}\right)_i^2\bigg)^{1/2}\\
&\lesssim_p \sum_{i\in I_s} |(\Delta_s f)_i| |(\pi_{s+1}f)_i|^{p-1}+u 2^{s/2}\bigg(\sum_{i\in I_s^c}(\Delta_s f)_i^2(\pi_{s+1}f)_i^{2(p-1)}\bigg)^{1/2}\\
&\quad +\sum_{i\in I_s} |(\Delta_s f)_i| |(\pi_{s}f)_i|^{p-1}
+u 2^{s/2} \bigg(\sum_{i\in I_s^c}(\Delta_s f)_i^2 (\pi_{s}f)_i^{2(p-1)}\bigg)^{1/2}\\
&=: T_1^{(s)} + T_2^{(s)},
\end{align*}
where in the inequality ($\star$) we used that $|x^p-y^p|\le p|x-y|(|x|^{p-1}+|y|^{p-1})$, which follows directly from applying the mean value theorem to the function $h(x)=x^p$: for $x>y$, there exists some $\xi$ with $y\le \xi\le x$ such that $x^p-y^p=(x-y)h^{\prime}(\xi)=p (x-y) \xi^{p-1},$
thus $|x^p-y^p|=p|x-y||\xi|^{p-1}\le p|x-y|\max\{|x|^{p-1},|y|^{p-1}\}\le p|x-y|(|x|^{p-1}+|y|^{p-1}).$

Let $(j_s)_{s\ge s_0}$ be the sequence defined in \eqref{eq:j_s}. For each $s\ge s_0$, we choose $I_s$ to be the union of the $j_s-1$ largest coordinates of $\left(\left|(\Delta_s f)_i \right|\right)_{i=1}^N$, the $j_s-1$ largest coordinates of $\left(\left|(\pi_{s+1} f)_i\right|\right)_{i=1}^N$, and the $j_s-1$ largest coordinates of $\left(\left|(\pi_{s} f)_i\right|\right)_{i=1}^N$.  For a vector $x=(x_i)_{i=1}^{N}\in \R^{N}$, we denote by $\left(x_i^*\right)_{i=1}^N$  the non-increasing rearrangement of $\left(\left|x_i\right|\right)_{i=1}^N$. By Cauchy-Schwarz inequality,
\begin{align*}
T_1^{(s)} &=\sum_{i\in I_s} |(\Delta_s f)_i| |(\pi_{s+1}f)_i|^{p-1}+u 2^{s/2} \bigg(\sum_{i\in I_s^c}(\Delta_s f)_i^2(\pi_{s+1}f)_i^{2(p-1)}\bigg)^{1/2} \\
&\le 
\bigg(\sum_{i\in I_s} (\Delta_s f)^2_i\bigg)^{1/2}
\bigg(\sum_{i\in I_s} (\pi_{s+1}f)^{2(p-1)}_i \bigg)^{1/2}
+u2^{s/2} \bigg(\sum_{i\in I_s^c} (\Delta_{s} f)^4_i \bigg)^{1/4} \bigg(\sum_{i\in I_s^c} (\pi_{s+1}f)^{4(p-1)}_i \bigg)^{1/4}\\
&\lesssim \bigg(\sum_{i<j_s} \left[\left(\Delta_s f\right)^2\right]^*_i\bigg)^{1/2} 
\sup_{f\in \F} \bigg(\sum_{i=1}^{N} f^{2(p-1)}_i\bigg)^{1/2}
+u 2^{s/2} \bigg(\sum_{i\ge j_s} \left[(\Delta_{s} f)^4\right]^{*}_i \bigg)^{1/4}\bigg(\sum_{i\ge j_s} \left[(\pi_{s+1}f)^{4(p-1)}\right]^{*}_i \bigg)^{1/4},
\end{align*}
where in the last step we have used that $ \sum_{i\in I_s} (\Delta_s f)^2_i\le 3 \sum_{i<j_s} \left[\left(\Delta_s f\right)^2\right]^*_i $ since $|I_{s}|\le 3(j_s-1)$, and
\begin{align*}
    \bigg(\sum_{i\in I_s} (\pi_{s+1}f)^{2(p-1)}_i \bigg)^{1/2}
    \le 
    \bigg(\sum_{i=1}^N (\pi_{s+1}f)^{2(p-1)}_i \bigg)^{1/2} 
    \le 
    \sup_{f\in \F}\bigg(\sum_{i=1}^{N} f_i^{2(p-1)}\bigg)^{1/2}.
\end{align*}
A similar argument shows that 
\begin{align*}
T_{2}^{(s)}&=\sum_{i\in I_s} |(\Delta_s f)_i| |(\pi_{s}f)_i|^{p-1}
+u2^{s/2} \bigg(\sum_{i\in I_s^c}(\Delta_s f)_i^2 (\pi_{s}f)_i^{2(p-1)}\bigg)^{1/2}\\
&\lesssim
\bigg(\sum_{i<j_s} \left[\left(\Delta_s f\right)^2\right]^*_i\bigg)^{1/2} 
\sup_{f\in \F}\bigg(\sum_{i=1}^{N} f^{2(p-1)}_i \bigg)^{1/2}
+u 2^{s/2} 
\bigg(\sum_{i\ge j_s} \left[(\Delta_{s} f)^4\right]^{*}_i \bigg)^{1/4}
\bigg(\sum_{i\ge j_s} \left[(\pi_{s}f)^{4(p-1)}\right]^{*}_i \bigg)^{1/4}.
\end{align*}

Further, repeating this argument for the last term in \eqref{eq:main2_aux1} gives that, with probability at least $1-2\exp(-u^2 2^{s_0-1})$,
\begin{align*}
&\bigg|\sum_{i=1}^N \varepsilon_i (\pi_{s_0} f)^p_i \bigg|
\le \bigg(\sum_{i<j_{s_0}} \left[\left(\pi_{s_0} f\right)^2\right]^*_i\bigg)^{1/2}\sup_{f\in \F}\bigg(\sum_{i=1}^{N} f^{2(p-1)}_i\bigg)^{1/2}
+u2^{s_0/2} \bigg(\sum_{i\ge j_{s_0}} \left[(\pi_{s_0}f)^{2p}\right]_i^{*}\bigg)^{1/2}\\
&\le \bigg(\sum_{i<j_{s_0}} \left[\left(\pi_{s_0} f\right)^2\right]^*_i\bigg)^{1/2}\sup_{f\in \F}\bigg(\sum_{i=1}^{N} f^{2(p-1)}_i\bigg)^{1/2}+u2^{s_0/2}N^{\frac{p-2}{4(p-1)}
}\bigg(\sum_{i\ge j_{s_0}} \left[(\pi_{s_0}f)^{4(p-1)}\right]^{*}_i\bigg)^{\frac{p}{4(p-1)}},
\end{align*}
where the last step follows by H\"older's inequality and the fact that $|N-j_{s_0}+1|^{\frac{p-2}{4(p-1)}} \le N^{\frac{p-2}{4(p-1)}}.$

Observe that $\left|\left\{\pi_{s} f: f \in \F \right\}\right| \le |\F_{s}| \le 2^{2^{s}}$, $\left|\left\{\pi_{s+1} f: f \in \F \right\}\right| \le |\F_{s+1}| \le 2^{2^{s+1}}$ and $\left|\left\{\Delta_s f: f \in \F \right\}\right| \le |\F_s|\cdot |\F_{s+1}|\le 2^{2^s}\cdot 2^{2^{s+1}}\le 2^{2^{s+2}}$. For $u\ge 4$, we sum the probabilities for all $s\ge s_0$ to obtain that, with $(\varepsilon_i)_{i=1}^N$ probability at least
\begin{align*}
1-2\exp(-u^2 2^{s_0-1})\cdot 2^{2^{s_0}} -\sum_{s\ge s_0} 2\exp(-u^2 2^{s-1})\cdot 2^{2^{s+2}}\ge 1-2\exp(-u^2 2^{s_0}/8),
\end{align*}
 for every $f\in \F$,
\begin{align*}
&\bigg|\sum_{i=1}^N \varepsilon_i f^p(X_i)\bigg|
 \le \sum_{s\ge s_0} \bigg|\sum_{i=1}^N \varepsilon_i ((\pi_{s+1}f)^p-(\pi_s f)^p)(X_i)\bigg|+\bigg|\sum_{i=1}^N \varepsilon_i (\pi_{s_0} f)^p(X_i)\bigg|\\
&\lesssim_p \sum_{s\ge s_0} \big(T_1^{(s)}+T_{2}^{(s)}\big)+\bigg|\sum_{i=1}^N \varepsilon_i (\pi_{s_0} f)^p(X_i)\bigg|  \\
&\lesssim_p \sup_{f\in \F}\bigg(\sum_{i=1}^{N} f^{2(p-1)}_i\bigg)^{1/2}\Bigg(\sum_{s\ge s_0}\bigg(\sum_{i<j_s} \left[\left(\Delta_s f\right)^2 \right]^*_i \bigg)^{1/2}+\bigg(\sum_{i<j_{s_0}} \left[\left(\pi_{s_0} f\right)^2 \right]^*_i \bigg)^{1/2}\Bigg)\\
&\quad +u \left(\sum_{s\ge s_0}2^{s/2} \bigg(\sum_{i\ge j_s} \left[(\Delta_{s} f)^4\right]^{*}_i \bigg)^{1/4} 
\left(
 \bigg(\sum_{i\ge j_s} \left[(\pi_{s+1}f)^{4(p-1)}\right]^{*}_i\bigg)^{1/4}
 + \bigg(\sum_{i\ge j_s} \left[(\pi_{s}f)^{4(p-1)}\right]^{*}_i\bigg)^{1/4}
 \right)
\right)\\
&\quad + u 2^{s_0/2}N^{\frac{p-2}{4(p-1)}
}\bigg(\sum_{i\ge j_{s_0}} \left[(\pi_{s_0}f)^{4(p-1)}\right]^{*}_i\bigg)^{\frac{p}{4(p-1)}}.
\end{align*}

On the other hand, we know from Corollary \ref{coro:1} that with $(X_i)_{i=1}^{N}$ probability at least $1-2\exp(-c_1u^2 2^{s_0})$, the four inequalities in Corollary \ref{coro:1} hold for every $f \in \F$ and $s\ge s_0$ simultaneously. Therefore, combining the $(\varepsilon_i)_{i=1}^N$ probability and the $(X_i)_{i=1}^N$ probability implies that there exists a constant $\tilde{c}(p)$ depending only on $p$ such that, for $u\ge c(p)\lor 4$, it holds with probability at least
\[
1-2\exp(-u^2 2^{s_0}/8)-2\exp(-c_1u^2 2^{s_0})\ge 1-2 \exp \left(-\tilde{c}(p) u^2 2^{s_{0}}\right)
\]
that, for every $f\in \F$,
\begin{align*}
\bigg|\sum_{i=1}^N \varepsilon_i f^p_i\bigg|
&\lesssim_p \sup_{f\in\F}\bigg(\sum_{i=1}^N f^{2(p-1)}_i\bigg)^{1/2} u \bigg(\sum_{s\ge s_0}2^{s/2}\|\Delta_s f\|_{(c_3 u^2 2^s)}+2^{s_0/2}\|\pi_{s_0}f\|_{(c_3 u^2 2^{s_0})}\bigg)\\
&\quad + u \bigg(\sum_{s\ge s_0} 2^{s/2} N^{1/4} \|\Delta_s f\|_{L_8} N^{1/4} \big(\|\pi_{s}f\|_{L_{8(p-1)}}^{p-1}+\|\pi_{s+1}f\|_{L_{8(p-1)}}^{p-1}\big)\bigg)\\
&\quad +u 2^{s_0/2}N^{\frac{p-2}{4(p-1)}} \left(N^{\frac{1}{4(p-1)}}\|\pi_{s_0} f\|_{L_{8(p-1)}}\right)^{p}\\
&\overset{\text{($\star$)}}{\lesssim}_p \sup_{f\in\F}\bigg(\sum_{i=1}^N f^{2(p-1)}_i\bigg)^{1/2} u \widetilde{\Lambda}_{s_0,cu}(\F) + u  N^{1/2}d^{p-1}_{\psi_2}(\F)\bigg(\sum_{s\ge s_0} 2^{s/2}\|\Delta_s f\|_{\psi_2}\bigg)+u 2^{s_0/2}N^{1/2} d^p_{\psi_2}(\F) \\
&\overset{\text{($\star\star$)}}{\lesssim}_p \sup_{f\in\F}\bigg(\sum_{i=1}^N f^{2(p-1)}_i\bigg)^{1/2} u \widetilde{\Lambda}_{s_0,cu}(\F) + u  N^{1/2}d^{p-1}_{\psi_2}(\F)\left(\gamma(\F,\psi_2)+2^{s_0/2}d_{\psi_2}(\F)\right),
\end{align*}
where ($\star$) follows by choosing an almost optimal admissible sequence in $\F$ and using the definition of $\widetilde{\Lambda}_{s_0,u}(\F)$, and the fact that $\|\pi_{s}f\|_{L_{8(p-1)}}^p\lesssim_p \|\pi_{s}f\|_{\psi_2}^p\le d_{\psi_2}^p(\F), \|\pi_{s_0}f\|_{L_{8(p-1)}}^p\lesssim_p \|\pi_{s_0} f\|_{\psi_2}^p \le d_{\psi_2}^p(\F)$; ($\star\star$) follows by
\begin{align*}
\sum_{s\ge s_0} 2^{s/2}\|\Delta_s f\|_{\psi_2}&\le \sum_{s\ge s_0} 2^{s/2}(\|f-\pi_{s}f\|_{\psi_2}+\|f-\pi_{s+1}f\|_{\psi_2}) \lesssim \gamma(\F,\psi_2).
\end{align*}

Now we apply Theorem~\ref{thm:l_m_norm} and Remark \ref{rem:lm_norm} with $m=2(p-1)\ge 2$ to obtain the high-probability bound: for any $w>0$, it holds with probability at least $1-2\exp(-\min\{ (\sqrt{N}w)^{1/(p-1)},w^2\})$ that
\[
\sup_{f\in \F} \bigg(\sum_{i=1}^{N}|f|^{2(p-1)}_i\bigg)^{1/2}
\lesssim_p \gamma^{p-1}(\F,\psi_2)+ N^{1/2}d^{p-1}_{\psi_2}(\F)+N^{1/4}d^{p-1}_{\psi_2}(\F)\sqrt{w}.
\]
Moreover, by Lemma \ref{lemma:Lambda_Gamma}, $ \Lambda_{s_0, u}(\F) \lesssim \gamma\left(\F, \psi_2\right)$, thus
\[
\widetilde{\Lambda}_{s_0,cu}(\F)=\Lambda_{s_0, u}(\F)+2^{s_0 / 2} \sup_{f \in \F}\left\|\pi_{s_0} f\right\|_{(u^2 2^{s_0})}\lesssim \gamma(\F,\psi_2)+2^{s_0/2}d_{\psi_2}(\F).
\]
Consequently, we have shown that there exists constants $c(p)$ and $\tilde{c}(p)$ such that for any $u\ge c(p)\lor 4$, it holds with probability at least $1-2\exp(-\tilde{c}(p)u ^2 2^{s_0})-2\exp(-\min \{(\sqrt{N}w)^{1/(p-1)},w^2\})$ that
\begin{align*}
\sup_{f\in \F}\bigg|\sum_{i=1}^N \varepsilon_i f^p(X_i)\bigg| & \lesssim_p \sup_{f\in\F}\bigg(\sum_{i=1}^N f^{2(p-1)}(X_i)\bigg)^{1/2} u \widetilde{\Lambda}_{s_0,cu}(\F) + u  N^{1/2}d^{p-1}_{\psi_2}(\F)\left(\gamma(\F,\psi_2)+2^{s_0/2}d_{\psi_2}(\F)\right)\\
&\lesssim_p u \left(\gamma(\F,\psi_2)+2^{s_0/2}d_{\psi_2}(\F)\right)\left(\gamma^{p-1}(\F,\psi_2)+ N^{1/2}d^{p-1}_{\psi_2}(\F)+N^{1/4}d^{p-1}_{\psi_2}(\F)\sqrt{w}\right).
\end{align*}

For $u\ge c(p)\lor 4$, we set $s_0$ such that $2^{s_0}\asymp \left(\gamma(\F,\psi_{2})/d_{\psi_2}(\F)\right)^2$ and take 
\[
w=\left(\tilde{c}(p)u^{2}2^{s_0}\right)^{1/2}\lor \frac{(\tilde{c}(p)u^2 2^{s_0})^{p-1}}{\sqrt{N}}\asymp_p u \frac{\gamma(\F,\psi_2)}{d_{\psi_2}(\F)}+\frac{1}{\sqrt{N}}u^{2(p-1)}\left(\frac{\gamma(\F,\psi_2)}{d_{\psi_2}(\F)}\right)^{2(p-1)}.
\]
Then, with probability at least $1-4\exp(-c^{\prime}(p)u^2 (\gamma(\F,\psi_{2})/d_{\psi_2}(\F))^2),$
\begin{align*}
\sup_{f\in \F}\bigg|\frac{1}{N}\sum_{i=1}^N \varepsilon_i f^p(X_i)\bigg| & \lesssim_p \frac{u}{N} \left(\gamma(\F,\psi_2)+2^{s_0/2}d_{\psi_2}(\F)\right)\left(\gamma^{p-1}(\F,\psi_2)+ N^{1/2}d^{p-1}_{\psi_2}(\F)+N^{1/4}d^{p-1}_{\psi_2}(\F)\sqrt{w}\right)\\
&\lesssim_p u\left(\frac{\gamma(\F,\psi_2)d_{\psi_2}^{p-1}(\F)}{\sqrt{N}}+\frac{\gamma^{p}(\F,\psi_2)}{N}\right) +u^{3/2}\frac{\gamma^{3/2}(\F,\psi_2)d^{p-3/2}_{\psi_2}(\F)}{N^{3/4}}+u^p \frac{\gamma^p(\F,\psi_2)}{N}\\
&\lesssim u\frac{\gamma(\F,\psi_2)d_{\psi_2}^{p-1}(\F)}{\sqrt{N}} +u^{3/2}\frac{\gamma^{3/2}(\F,\psi_2)d^{p-3/2}_{\psi_2}(\F)}{N^{3/4}}+u^p \frac{\gamma^p(\F,\psi_2)}{N},
\end{align*}
where in the last step we used that $u\le u^p$ since $u\ge 1$ and $p\ge 2$. Observe that for $p\ge 2$,
\[
u^{3/2}\frac{\gamma^{3/2}(\F,\psi_2)d^{p-3/2}_{\psi_2}(\F)}{N^{3/4}}\le u\frac{\gamma(\F,\psi_2)d_{\psi_2}^{p-1}(\F)}{\sqrt{N}}+u^p \frac{\gamma^p(\F,\psi_2)}{N},
\]
i.e., $t^{3/2}\le t+t^p N^{p/2-1}$ with $t=\frac{u\gamma(\F,\psi_2)}{\sqrt{N}d_{\psi_2}(\F)}\ge 0$. Therefore, for any $u\ge c(p)\lor 4$, it holds with probability  at least $1-4\exp(-c^{\prime}(p)u^2 (\gamma(\F,\psi_{2})/d_{\psi_2}(\F))^2)$ that
\[
\sup_{f\in \F}\bigg|\frac{1}{N}\sum_{i=1}^N \varepsilon_i f^p(X_i)\bigg| \lesssim_p  u\frac{\gamma(\F,\psi_2)d_{\psi_2}^{p-1}(\F)}{\sqrt{N}}+u^p \frac{\gamma^p(\F,\psi_2)}{N}.
\]

By the symmetrization inequality for empirical processes \cite[Theorem 1.14]{mendelson2016upper}, see also \cite{gine1984some} and \cite[Chapter 2.3]{van2023weak}, we have with probability  at least $1-C\exp(-c^{\prime}(p)u^2 (\gamma(\F,\psi_{2})/d_{\psi_2}(\F))^2)$ that
\[
\sup_{f\in \F}\bigg|\frac{1}{N}\sum_{i=1}^N f^p(X_i)-\E f^p(X)\bigg| \lesssim_p  u\frac{\gamma(\F,\psi_2)d_{\psi_2}^{p-1}(\F)}{\sqrt{N}}+u^p \frac{\gamma^p(\F,\psi_2)}{N},
\]
where $C>4$ is some absolute constant. Finally, the high-probability bound in Theorem \ref{thm:main2} follows from a simple and standard modification using Lemma \ref{lemma:simple_tail}. Furthermore, integrating out the tail bound (note that $\gamma(\F,\psi_2)\gtrsim d_{\psi_2}(\F)$) gives that
\[
\E \sup_{f\in \F} \bigg|\frac{1}{N}\sum_{i=1}^N f^p(X_i) -\E f^p(X)\bigg|\lesssim_p \frac{\gamma(\F,\psi_2)d^{p-1}_{\psi_2}(\F)}{\sqrt{N}}+\frac{\gamma^{p}(\F,\psi_2)}{N},
\]
which completes the proof.
\end{proof}

We conclude this subsection with two lemmas that were used in the proof of Theorem~\ref{thm:main2}. The proofs can be found in Appendix \ref{app:A}.

\begin{lemma}\label{lemma:gamma_2_d_psi_2}
If $0\in \F$ or $\F$ is symmetric, then $\gamma(\mcF,\psi_2) \gtrsim d_{\psi_2}(\mcF).$
 \end{lemma}

\begin{lemma}\label{lemma:simple_tail}
Suppose that for any $u\ge C_1(p)$, it holds with probability at least $1-C_2(p)\exp(-c_3(p) u^2 \omega)$ that
\[
(*)\lesssim_p uA+u^p B,
\]
where $(*)$ is some random variable, $C_1(p)\ge 1, C_2(p)\ge 1, c_3(p)> 0, \omega\ge 1, A\ge 0, B\ge 0$ are constants. Then, for any $u \ge 1$, it holds with probability at least $1-\exp(-u^2 \omega )$ that
\[
(*)\lesssim_p uA+u^{p}B.
\]
\end{lemma}

\subsection{Proof of Theorem~\ref{thm:l_m_norm}}\label{subsec:l_m_norm}

This subsection presents the proof of Theorem~\ref{thm:l_m_norm}, which establishes a uniform bound on the $\ell_{m}$-norm of the coordinate projection vector of the function class $\F$. The proof leverages generic chaining and builds upon ideas from Bednorz \cite[Section 5]{bednorz2014} for analyzing quadratic processes, while also addressing new challenges that arise when handling arbitrary $m\ge 1$. A key component of our approach is the use of a different Young function, defined as $\phi(x)=2^{\min\{(\sqrt{N}x)^{2/m},\, x^2\}}-1$, along with $\alpha$-sub-exponential inequalities \cite{gotze2021concentration} to precisely characterize the tail behavior. This leads to a distinct cut-off point in the chaining process ---depending on $m$--- that separates sub-Gaussian and sub-exponential regimes. Similar cut-off phenomena between these two regimes first appeared in the analysis of quadratic empirical processes (see, e.g., \cite[Section 5]{mendelson2010empirical}, \cite[Theorem 5.5]{dirksen2015tail} and \cite[Theorem 14.2.6]{talagrand2022upper}). However, our primary focus here is to control the non-centralized $\ell_m$-norm of the coordinate projection vector for any $m\ge 1$.

\begin{proof}[Proof of Theorem~\ref{thm:l_m_norm}]

Let $(\mcF_s)_{s\ge 0}\subset \F$ be an increasing sequence which satisfies $\F_s \subset \F_{s+1}, \left|\F_0\right|=1, \left|\F_s\right| \leq 2^{2^s}-1$ for $s \ge 1,$ and $\bigcup_{s=0}^{\infty}\F_s$ is dense in $\F$. Note that here we slightly change the requirement of the size of $\F_s$ (see Definition \ref{def:admissible_sequence_gamma_functional}) for technical reasons. For any $f\in \F$ and $s\ge 0$, $\pi_s f$ is the nearest point in $\F_s$ to $f$ with respect to the $\psi_2$-norm. We further assume that $(\F_{s})_{s\ge 0}$ is an almost optimal admissible sequence, ensuring that
\[
\sup_{f\in \F}\sum_{s=0}^{\infty} 2^{s/2} \|f-\pi_s f\|_{\psi_2}\lesssim \gamma(\F,\psi_2).
\]

We first focus on the case $m>1$. Let $N_s:=2^{2^{s}}-1$ for $s\ge 1$, and let $s^{*}$ be the integer that satisfies $2^{(m-1)s^{*}}\le N<2^{(m-1)(s^{*}+1)}$. Recall that $\phi(x)=2^{\min\{(\sqrt{N}x)^{2/m},\, x^2\}}-1$, and its inverse function is
\[
\phi^{-1}(x)=\max\left\{\sqrt{\log_2(1+x)},\sqrt{\frac{(\log_2(1+x))^{m}}{N}}\right\},\quad x\ge 0.
\]
Then,
\begin{align*}
    \phi^{-1}(N_{s})=\phi^{-1}(2^{2^{s}}-1)=\begin{cases}
        \big(\frac{2^{sm}}{N}\big)^{1/2} & \text{ if } \ s>s^{*},\\
        2^{s/2} & \text{ if } \ s\le s^{*}.
    \end{cases}
\end{align*}

We prove the conclusion by establishing the following two claims:

\textbf{Claim I.}
\begin{align*}
\sup_{f\in \F}\bigg(\sum_{i=1}^N\left|f-\pi_{s^{*}} f\right|^m\left(X_i\right)\bigg)^{1/m}\lesssim_m  \gamma(\F,\psi_2)+N^{1/m}d_{\psi_2}(\F)+N^{1/2m}d_{\psi_2}(\F) (\phi^{-1}(Z_1))^{1/m},
\end{align*}
where $Z_1 \ge 0$ and $\E \, Z_1\le 3$.

\textbf{Claim II.}
\begin{align*}
\sup_{f\in \F} \bigg(\sum_{i=1}^{N} |\pi_{s^{*}} f|^m (X_i)\bigg)^{1/m} &\lesssim_m N^{1/m} d_{\psi_2}(\F)+ N^{1/2m}d^{(m-1)/m}_{\psi_2}(\F)\gamma^{1/m}(\F,\psi_2)\\
&\quad +N^{1/2m} d_{\psi_2}(\F)\left((\phi^{-1}(Z_2))^{1/m}+(\phi^{-1}(Z_3))^{1/m}\right),
\end{align*}
where $Z_2\ge 0, Z_3 \ge 0$ and $\E \, Z_2\le 3, \E \, Z_3\le 3$.

Claim I and Claim II together with the triangle inequality imply that
\begin{align*}
   & \sup_{f\in \F} \bigg(\sum_{i=1}^N |f|^m (X_i)\bigg)^{1/m} \le\sup_{f\in \F}\bigg(\sum_{i=1}^N\left|f-\pi_{s^{*}} f\right|^m\left(X_i\right)\bigg)^{1/m}+\sup_{f\in \F} \bigg(\sum_{i=1}^{N} |\pi_{s^{*}} f|^m (X_i)\bigg)^{1/m} \\
    &\lesssim_m \gamma(\F,\psi_2)+N^{1/m}d_{\psi_2}(\F)+N^{1/2m}d^{(m-1)/m}_{\psi_2}(\F)\gamma^{1/m}(\F,\psi_2)\\
    &\quad +N^{1/2m}d_{\psi_2}(\F) \left((\phi^{-1}(Z_1))^{1/m}+(\phi^{-1}(Z_2))^{1/m}+(\phi^{-1}(Z_3))^{1/m}\right)\\
    &\lesssim_m \gamma(\F,\psi_2)+N^{1/m}d_{\psi_2}(\F)+N^{1/2m}d^{(m-1)/m}_{\psi_2}(\F)\gamma^{1/m}(\F,\psi_2)+N^{1/2m}d_{\psi_2}(\F) (\phi^{-1}(Z_1+Z_2+Z_3))^{1/m}\\
    &=\gamma(\F,\psi_2)+N^{1/m}d_{\psi_2}(\F)+N^{1/2m}d^{(m-1)/m}_{\psi_2}(\F)\gamma^{1/m}(\F,\psi_2)+N^{1/2m}d_{\psi_2}(\F) (\phi^{-1}(Z))^{1/m},
\end{align*}
where $Z:=Z_1+Z_2+Z_3$ satisfies $\E \, Z=\E \, Z_1+\E \, Z_2+\E \, Z_3\le 9$. This proves the first inequality in Theorem~\ref{thm:l_m_norm}.

For $m \ge 2$, note that 
\[
N^{1/2m}d^{(m-1)/m}_{\psi_2}(\F)\gamma^{1/m}(\F,\psi_2)\le \gamma(\F,\psi_2)+N^{1/m}d_{\psi_2}(\F),
\]
or, equivalently,
\[
N^{1/2m}\left(\frac{\gamma(\F,\psi_2)}{d_{\psi_2}(\F)}\right)^{1/m}\le \frac{\gamma(\F,\psi_2)}{d_{\psi_2}(\F)}+N^{1/m},
\]
which can be verified by noting that $xy\le x^2+y^m$ where $x=N^{1/2m}\ge 1$ and $y=\left(\frac{\gamma(\F,\psi_2)}{d_{\psi_2}(\F)}\right)^{1/m}\ge 0$. This completes the proof.

When $m=1$, there is no $s^*$ satisfying $2^{(m-1)s^{*}}\le N<2^{(m-1)(s^{*}+1)}$. However, in this case, we have $\phi(x)=2^{\min\{Nx^{2},x^2\}}-1 \equiv 2^{x^2}-1$ and its inverse function is given by $\phi^{-1}(x)\equiv \sqrt{\log_2(1+x)}$. As will become clear in the proof of \textbf{Claim II} below, \textbf{Claim II} remains valid for \emph{any} integer $s^{*}$ in this setting. By taking $s^{*}\to \infty$, we obtain
\begin{align*}
\sup_{f\in \F} \sum_{i=1}^N |f(X_i)|&=\sup_{f\in \F}\lim_{s^*\to \infty} \sum_{i=1}^{N}|\pi_{s^*} f|(X_i)\le \lim_{s^{*}\to \infty}\sup_{f\in \F} \sum_{i=1}^{N}|\pi_{s^*} f|(X_i)\\
&\lesssim N d_{\psi_2}(\F)+ N^{1/2}\gamma(\F,\psi_2)+N^{1/2} d_{\psi_2}(\F)\left((\phi^{-1}(Z_2))+(\phi^{-1}(Z_3))\right)\\
&\le N d_{\psi_2}(\F)+ N^{1/2}\gamma(\F,\psi_2)+N^{1/2} d_{\psi_2}(\F)\phi^{-1}(Z),
\end{align*}
where $Z:=Z_2+Z_3$ satisfies $\E \, Z=\E \, Z_2+\E \, Z_3\le 6$. Consequently,
\[
\sup_{f\in \F} \frac{1}{N}\sum_{i=1}^{N}|f\left(X_i\right)|\lesssim d_{\psi_2}(\F)+\frac{\gamma(\F,\psi_2)}{\sqrt{N}}+\frac{d_{\psi_2}(\F) \phi^{-1}(Z)}{\sqrt{N}},
\]
as desired.

Now we prove \textbf{Claim I} and \textbf{Claim II}.

\textbf{Proof of Claim I:}
We use the following chaining
\[
\left(f-\pi_{s^{*}} f\right)\left(X_i\right)=\sum_{k=0}^{\infty}\left(\pi_{s_{k+1}} f-\pi_{s_k} f\right)\left(X_i\right),\quad 1\le i \le N, 
\]
where the sequence $\left\{s_k\right\}_{k=0}^{\infty}$ depends on $f$ and is defined as:
$s_0=s^{*}$ and
\[
s_k:=\inf \left\{s>s_{k-1}: 2 \|f- \pi_s f\|_{\psi_2}< \|f- \pi_{s_{k-1}}f\|_{\psi_2}\right\},\quad k \ge 1.
\]
Let $\Omega_k(f)=\frac{1}{N} \sum_{i=1}^N\left|\pi_{s_{k+1}} f-\pi_{s_k} f\right|^m\left(X_i\right)$. By the triangle inequality,
\[
\bigg(\frac{1}{N} \sum_{i=1}^N\left|f-\pi_{s^{*}} f\right|^m\left(X_i\right)\bigg)^{1/m} \le \sum_{k=0}^{\infty}\left(\Omega_k(f)\right)^{1/m}.
\]
Moreover, since $m> 1$,
\[
\left(\Omega_k(f)\right)^{1/m} \le\left(\left|\Omega_k(f)-\E \,\Omega_k(f)\right|\right)^{1/m}+\left(\E \, \Omega_k(f)\right)^{1/m}.
\]
Clearly,  
\[
\left(\E\, \Omega_k(f)\right)^{1/m}=\|\pi_{s_{k+1}} f-\pi_{s_{k}} f\|_{L_m}  \lesssim_m \|\pi_{s_{k+1}} f- \pi_{s_{k}} f\|_{\psi_2}.
\]

Observe that
\[
\phi\bigg(\frac{ \sqrt{N}|\Omega_k(f)-\E\, \Omega_k(f)|}{C(m)  \|\pi_{s_{k+1}} f-\pi_{s_k} f\|^m_{\psi_2}}\bigg)\le \sum_{g,h\in \F_{s_{k+1}}}\phi\left(|\mathsf{X_1}(g,h)|\right), 
\]
where 
\[
\mathsf{X_1}(g,h):= \frac{\sum_{i=1}^N (|g-h|^m -\E|g-h|^{m})(X_i)}{C(m) \sqrt{N}\|g-h\|^m_{\psi_2}}.
\]
By Lemma \ref{lemma:alpha-subexp-2}, we have $\E \, \phi(|\mathsf{X_1}(g,h)|)\le 3$ for any $g,h\in \F$.
Therefore,
\begin{align*}
\frac{\sqrt{N}|\Omega_k(f)-\E\, \Omega_k(f)|}{C(m) \|\pi_{s_{k+1}} f-\pi_{s_k} f\|^m_{\psi_2}} &\le \phi^{-1}\bigg(\sum_{g,h\in \F_{s_{k+1}}}\phi\left(|\mathsf{X_1}(g,h)|\right)\bigg)\\
&=\phi^{-1}\bigg(N_{s_{k+1}}^3 \cdot \frac{1}{N_{s_{k+1}}^3}\sum_{g,h\in \F_{s_{k+1}}}\phi\left(|\mathsf{X_1}(g,h)|\right)\bigg)\\
&\overset{\text{($\star$)}}{\lesssim}_m \phi^{-1}(N_{s_{k+1}})+\phi^{-1}\bigg(\frac{1}{N_{s_{k+1}}^3}\sum_{g,h\in \F_{s_{k+1}}}\phi\left(|\mathsf{X_1}(g,h)|\right)\bigg)\\
&\le \phi^{-1}(N_{s_{k+1}})+\phi^{-1}(Z_1),
\end{align*}
where ($\star$) follows from Lemma \ref{lemma:property_phi} and $
\phi^{-1}(N_{s_{k+1}}^3)\lesssim_m \phi^{-1}(N_{s_{k+1}}^2)+\phi^{-1}(N_{s_{k+1}})\lesssim_m 3\phi^{-1}(N_{s_{k+1}}).$
  In the last inequality,
\[
Z_1:=\sum_{k=1}
^{\infty} \frac{1}{N_{k}^3} \sum_{g,h\in \F_k} \phi(|\mathsf{X_1}(g,h)|)\ge 0
\]
satisfies $\E\, Z_1\le 3\sum_{k=1}^{\infty} \frac{1}{N_{k}^3}\cdot N_k^2=3\sum_{k=1}^{\infty}\frac{1}{N_k}\le 3$.

Now we have
\[
|\Omega_k(f)-\E\, \Omega_k(f)|^{1/m}\lesssim_m N^{-1/2m} \|\pi_{s_{k+1}} f-\pi_{s_k} f\|_{\psi_2} \left((\phi^{-1}(N_{s_{k+1}}))^{1/m}+(\phi^{-1}(Z_1))^{1/m}\right).
\]
Since $s_{k+1}\ge s_1> s_0=s^{*}$, we have
\[
\phi^{-1}(N_{s_{k+1}})=\phi^{-1}(2^{2^{s_{k+1}}}-1)=\frac{1}{\sqrt{N}}2^{s_{k+1} m/2}.
\]
Then, 
\begin{align*}
    \left(\Omega_k(f)\right)^{1/m} &\le \left(\left|\Omega_k(f)-\E\, \Omega_k(f)\right|\right)^{1/m}+\left(\E \, \Omega_k(f)\right)^{1/m}\\
    &\lesssim_m N^{-1/2m} \|\pi_{s_{k+1}} f-\pi_{s_{k}} f\|_{\psi_2} \left( 2^{s_{k+1}/2}N^{-1/2m}+(\phi^{-1}(Z_1))^{1/m}\right) + \|\pi_{s_{k+1}} f- \pi_{s_{k}} f\|_{\psi_2}\\
    &=N^{-1/m} 2^{s_{k+1}/2} \|\pi_{s_{k+1}} f-\pi_{s_{k}} f\|_{\psi_2}+\|\pi_{s_{k+1}} f- \pi_{s_{k}} f\|_{\psi_2}\left(1+N^{-1/2m}(\phi^{-1}(Z_1))^{1/m}\right).
\end{align*}
We sum over $k$ to get
\begin{align}\label{eq:proposition_aux0}
    &\bigg(\frac{1}{N} \sum_{i=1}^N\left|f-\pi_{s^{*}} f\right|^m\left(X_i\right)\bigg)^{1/m} \le \sum_{k=0}^{\infty}\left(\Omega_k(f)\right)^{1/m}\nonumber\\
    &\lesssim_m N^{-1/m} \bigg(\sum_{k=0}^{\infty} 2^{s_{k+1}/2} \|\pi_{s_{k+1}} f-\pi_{s_{k}} f\|_{\psi_2}\bigg)+\bigg(\sum_{k=0}^{\infty}\|\pi_{s_{k+1}} f- \pi_{s_{k}} f\|_{\psi_2}\bigg)\left(1+N^{-1/2m}(\phi^{-1}(Z_1))^{1/m}\right)\nonumber\\
    &\overset{\text{($\star$)}}{\lesssim}_m N^{-1/m} \gamma(\F,\psi_2)+\|f-\pi_{s^{*}} f\|_{\psi_2}\left(1+N^{-1/2m}(\phi^{-1}(Z_1))^{1/m}\right)\nonumber\\
    &\lesssim N^{-1/m} \gamma(\F,\psi_2)+d_{\psi_2}(\F)\left(1+N^{-1/2m}(\phi^{-1}(Z_1))^{1/m}\right),
\end{align}
where in ($\star$) we used two facts
\begin{align}\label{eq:proposition_aux1}
\sum_{k=0}^{\infty} 2^{s_{k+1}/2} \|\pi_{s_{k+1}} f-\pi_{s_{k}} f\|_{\psi_2}\lesssim \gamma(\F,\psi_2),
\end{align}
and
\begin{align}\label{eq:proposition_aux2}
\sum_{k=0}^{\infty}\|\pi_{s_{k+1}} f- \pi_{s_{k}} f\|_{\psi_2}\lesssim \|f-\pi_{s^{*}} f\|_{\psi_2}.
\end{align}

The inequality \eqref{eq:proposition_aux1} follows by
\begin{align*}
    \sum_{k=0}^{\infty} 2^{s_{k+1}/2} \|\pi_{s_{k+1}} f-\pi_{s_{k}} f\|_{\psi_2}&\le \sum_{k=0}^{\infty} 2^{s_{k+1}/2} \left(\|\pi_{s_{k+1}} f-f\|_{\psi_2}+\|\pi_{s_k} f-f\|_{\psi_2}\right)\\
    &\overset{\text{($\star$)}}{\le} \sum_{k=0}^{\infty} 2^{s_{k+1}/2}\left(\|\pi_{s_{k+1}} f-f\|_{\psi_2}+2\|\pi_{s_{k+1}-1}f-f\|_{\psi_2}\right)\\
    &\le \sum_{n=s^{*}}^{\infty} 2^{n/2} \|f-\pi_n f\|_{\psi_2}+2\sqrt{2}\sum_{k=0}^{\infty}2^{(s_{k+1}-1)/2}\|f-\pi_{s_{k+1}-1}f\|_{\psi_2}\\
    &\le 4\sum_{n=0}^{\infty} 2^{n/2} \|f-\pi_n f\|_{\psi_2}\lesssim \gamma(\F,\psi_2),
\end{align*}
where ($\star$) follows by the definition of $\{s_k\}_{k=0}^{\infty}$: $\|\pi_{s_k} f-f\|_{\psi_2}\le 2 \|\pi_{s_{k+1}-1}f-f\|_{\psi_2}$. 

To prove the inequality \eqref{eq:proposition_aux2}, we observe that by the definition of $\{s_k\}_{k=0}^{\infty}$, it holds for every $k\ge 0$ that $\|f-\pi_{s_{k+1}} f\|_{\psi_2}\le \frac{1}{2}\|f-\pi_{s_k} f\|_{\psi_2}$. It is straightforward to show using induction that $\|f-\pi_{s_k} f\|_{\psi_2}\le \frac{1}{2^k}\|f-\pi_{s_0}f\|_{\psi_2}=\frac{1}{2^k}\|f-\pi_{s^{*}} f\|_{\psi_2}$. Therefore,
\begin{align*}
\sum_{k=0}^{\infty}\|\pi_{s_{k+1}} f-\pi_{s_{k}} f\|_{\psi_2}&\le \sum_{k=0}^{\infty}\|f-\pi_{s_k} f\|_{\psi_2}+\|f-\pi_{s_{k+1}} f\|_{\psi_2}\\
&\le \|f-\pi_{s^{*}} f\|_{\psi_2}\bigg(\sum_{k=0}^{\infty} \frac{1}{2^k}+\frac{1}{2^{k+1}}\bigg)\lesssim \|f-\pi_{s^{*}} f\|_{\psi_2}.
\end{align*}

Consequently, we have shown in \eqref{eq:proposition_aux0} that
\begin{align*}
\sup_{f\in \F}\bigg(\sum_{i=1}^N\left|f-\pi_{s^{*}} f\right|^m\left(X_i\right)\bigg)^{1/m}\lesssim_m  \gamma(\F,\psi_2)+N^{1/m}d_{\psi_2}(\F)+N^{1/2m}d_{\psi_2}(\F) (\psi^{-1}(Z_1))^{1/m},
\end{align*}
where $Z_1 \ge 0$ satisfies that $\E\, Z_1\le 3$. This completes the proof of \textbf{Claim I}.

\textbf{Proof of Claim II:}
For $f\in \F$, we define a sequence $\left\{s_k\right\}_{k=0}^{\infty}$ which depends on $f$ as follows: $s_0=0$ and
\[
s_k:=\min \left\{\inf \left\{s>s_{k-1}: 2 \|f- \pi_s f\|_{\psi_2}< \|f- \pi_{s_{k-1}} f\|_{\psi_2}\right\}, s^{*}\right\},\quad k\ge 1.
\]
Note that $s_k=s^{*}$ for a finite $k \ge 0$.

Let $\widetilde{\Omega}_k(f)=\frac{1}{N} \sum_{i=1}^N |\pi_{s_k} f|^m\left(X_i\right)-|\pi_{s_{k-1}}f|^m\left(X_i\right)$. We perform the chaining
\begin{align*}
& \bigg|\frac{1}{N} \sum_{i=1}^N|\pi_{s^{*}} f|^m(X_i)-\E|\pi_{s^{*}} f|^m-|\pi_0 f|^{m}(X_i)+\E|\pi_0 f|^m \bigg| \\
& \le  \sum_{k=1}^{\infty} \bigg|\frac{1}{N} \sum_{i=1}^N |\pi_{s_k} f|^m\left(X_i\right)-|\pi_{s_{k-1}}f|^m\left(X_i\right)-\E|\pi_{s_k} f|^m+\E|\pi_{s_{k-1}} f|^m\bigg|\\
&=: \sum_{k=1}^{\infty}\left|\widetilde{\Omega}_k(f)-\E\,\widetilde{\Omega}_k(f)\right|.
\end{align*}

Observe that
\[
\phi\bigg(\frac{\sqrt{N}|\widetilde{\Omega}_k(f)-\E\,\widetilde{\Omega}_k(f)|}{C(m) d_{\psi_2}^{m-1}(\F)\|\pi_{s_k} f-\pi_{s_{k-1}}f\|_{\psi_2}}\bigg)
\le 
\sum_{g,h\in \F_{s_k}}\phi\left(|\mathsf{X_2}(g,h)|\right), 
\]
where
\[
\mathsf{X_2}(g,h)=\frac{\sum_{i=1}^N (|g|^m-|h|^m-(\E |g|^m-\E|h|^m)) (X_i)}{C(m)\sqrt{N}d^{m-1}_{\psi_2}(\F)\|g-h\|_{\psi_2}}.
\]
By Lemma \ref{lemma:alpha-subexp-2}, we have $\E  \,\phi(|\mathsf{X_2}(g,h)|)\le 3$ for any $g,h\in \F$. Then, with the same technique as used in the proof of \textbf{Claim I},
\begin{align*}
\frac{\sqrt{N} |\widetilde{\Omega}_k(f)-\E\,\widetilde{\Omega}_k(f)|}{C(m) d_{\psi_2}^{m-1}(\F)\|\pi_{s_k} f-\pi_{s_{k-1}}f\|_{\psi_2}}&\le \phi^{-1}\bigg(\sum_{g,h\in \F_{s_k}}\phi\left(|\mathsf{X_2}(g,h)|\right)\bigg)\\
&\lesssim_m \phi^{-1}(N_{s_k})+\phi^{-1}(Z_2),
\end{align*}
where
\[
Z_2:=\sum_{k=1}^{\infty} \frac{1}{N_k^3} \sum_{g,h\in \F_k} \phi(|\mathsf{X_2}(g,h)|)
\]
satisfies $\E \, Z_2\le 3\sum_{k=1}^{\infty} \frac{1}{N_{k}^3}\cdot N_k^2=3\sum_{k=1}^{\infty}\frac{1}{N_k}\le 3$. Note that $\phi^{-1}(N_{s_k})=\phi^{-1}(2^{2^{s_k}}-1)=2^{s_k/2}$ since $s_k\le s^{*}$ by definition.
Then,
\begin{align*}
  &  \bigg|\frac{1}{N} \sum_{i=1}^N|\pi_{s^{*}} f|^m(X_i)-\E|\pi_{s^{*}} f|^m-|\pi_0 f|^{m}(X_i)+\E|\pi_0 f|^m \bigg| \le \sum_{k=1}^{\infty}\left|\widetilde{\Omega}_k(f)-\E\,\widetilde{\Omega}_k(f)\right|\\
    &\lesssim_m \sum_{k=1}^{\infty} N^{-1/2} d_{\psi_2}^{m-1}(\F)  \|\pi_{s_k} f-\pi_{s_{k-1}}f\|_{\psi_2} \left(2^{s_k/2}+\phi^{-1}(Z_2)\right)\\
    &\lesssim N^{-1/2} d_{\psi_2}^{m-1}(\F) \gamma(\F,\psi_2)+ N^{-1/2}d_{\psi_2}(\F)^{m}\phi^{-1}(Z_2),
\end{align*}
where in the last step we used that
\[
\sum_{k=1}^{\infty} 2^{s_k/2}\|\pi_{s_k} f-\pi_{s_{k-1}}f\|_{\psi_2}\lesssim \gamma(\F,\psi_2),\quad \sum_{k=1}^{\infty}\|\pi_{s_k} f-\pi_{s_{k-1}}f\|_{\psi_2}\lesssim \|f-\pi_0 f\|_{\psi_2}\lesssim d_{\psi_2}(\F).
\]
Moreover,
\[
\bigg|\frac{1}{N}\sum_{i=1}^{N} |\pi_0 f|^m(X_i)-\E |\pi_0 f|^m \bigg|\lesssim_m N^{-1/2}d_{\psi_2}^m(\F)\phi^{-1}(Z_3),
\]
where
\[
Z_3:=\phi\bigg(\frac{\sum_{i=1}^N (|\pi_{0} f|^m-\E |\pi_0 f|^m)(X_i)}{C(m)\sqrt{N}d^m_{\psi_2}(\F)}\bigg)
\] 
and $\E \, Z_3\le 3$ by Lemma \ref{lemma:alpha-subexp-2}.
Consequently,
\begin{align*}
\sup_{f\in \F} \bigg|\frac{1}{N}\sum_{i=1}^{N} |\pi_{s^{*}} f|^m (X_i)-\E |\pi_{s^{*}} f|^m (X) \bigg|\lesssim_m N^{-1/2}d^{m-1}_{\psi_2}(\F)\gamma(\F,\psi_2)+N^{-1/2} d^m_{\psi_2}(\F)(\phi^{-1}(Z_2)+\phi^{-1}(Z_3)).
\end{align*}
Since $\sup_{f\in \F}\E |\pi_{s^{*}} f|^m \lesssim_m d^m_{\psi_2}(\F) $, then
\[
\sup_{f\in \F} \frac{1}{N}\sum_{i=1}^{N} |\pi_{s^{*}} f|^m (X_i)\lesssim_m d^m_{\psi_2}(\F)+ N^{-1/2}d^{m-1}_{\psi_2}(\F)\gamma(\F,\psi_2)+N^{-1/2} d^m_{\psi_2}(\F)(\phi^{-1}(Z_2)+\phi^{-1}(Z_3)),
\]
and thus
\begin{align*}
\sup_{f\in \F} \bigg(\sum_{i=1}^{N} |\pi_{s^{*}} f|^m (X_i)\bigg)^{1/m} \nonumber&\lesssim_m N^{1/m} d_{\psi_2}(\F)+ N^{1/2m}d^{(m-1)/m}_{\psi_2}(\F)\gamma^{1/m}(\F,\psi_2)\\
&\quad +N^{1/2m} d_{\psi_2}(\F)\left((\phi^{-1}(Z_2))^{1/m}+(\phi^{-1}(Z_3))^{1/m}\right).
\end{align*}
This completes the proof of \textbf{Claim II}.
\end{proof}

We conclude this subsection with two lemmas that were used in the proof of Theorem~\ref{thm:l_m_norm}. The proofs can be found in Appendix \ref{app:B}.

\begin{lemma}\label{lemma:property_phi} For any $x,y\ge 0$, $\phi^{-1}(xy)\lesssim_m \phi^{-1}(x)+\phi^{-1}(y)$.
\end{lemma}
 
\begin{lemma}\label{lemma:alpha-subexp-2}
Let $X_1,\dots, X_N$ be independent random variables. There exists a constant $C(m)$ depending only on $m$ such that the following holds. For any $g,h\in \F$, we define
\[
\mathsf{X_1}(g,h)= \frac{\sum_{i=1}^N (|g-h|^m -\E|g-h|^{m})(X_i)}{C(m) \sqrt{N}\|g-h\|^m_{\psi_2}},\quad \mathsf{X_2}(g,h)=\frac{\sum_{i=1}^N (|g|^m-|h|^m-(\E |g|^m-\E|h|^m)) (X_i)}{C(m)\sqrt{N}d^{m-1}_{\psi_2}(\F)\|g-h\|_{\psi_2}},
\]
and
\[
\mathsf{X}_{3}(g)=\frac{\sum_{i=1}^N (|g|^m-\E |g|^m)(X_i)}{C(m)\sqrt{N}d^m_{\psi_2}(\F)}.
\]
It holds that
\[
\E\, \phi(|\mathsf{X_1}(g,h)|)\le 3,\quad \E \, \phi(|\mathsf{X_2}(g,h)|)\le 3,\quad \E \,\phi(|\mathsf{X_3}(g)|)\le 3.
\]
\end{lemma}

\section{Conclusions, discussion, and future directions}\label{sec:conclusions}

This paper has established sharp dimension-free concentration inequalities for simple random tensors. In so doing, we have developed the theory of $L_p$ empirical processes, leveraging generic chaining methods to obtain sharp high-probability upper bounds on their suprema. Our framework extends classical results on quadratic and product empirical processes to higher-order settings.

Several questions arise from this work. For Gaussian random vectors, Theorem \ref{thm:main1} provides matching upper and lower bounds for the deviation of the $p$-th order sample moment tensor from its expectation, up to a constant that depends only on $p$. In the special case of $p=2$, corresponding to the sample covariance, the recent work \cite{han2022exact} derived a sharp version of the Koltchinskii-Lounici theorem with optimal constants by exploiting the Gaussian min-max theorem, namely
\begin{align*}
\E \bigg\|\frac{1}{N}\sum_{i=1}^{N} X_i \otimes X_i -\mathbb{E}\, X \otimes X \bigg\| \le \|\Sigma\| \bigg(1+\frac{C}{\sqrt{r(\Sigma)}}\bigg)\bigg(2\sqrt{\frac{r(\Sigma)}{N}}+\frac{r(\Sigma)}{N}\bigg),
 \end{align*}
 where $C>0$ is some absolute constant. A natural question is whether similarly tight constants, which should depend on $p$, can be obtained for sample moment tensors of arbitrary order $p\ge 2$ in the Gaussian setting.

More recently, there has been growing interest in the estimation of smooth functionals of covariance operators \cite{koltchinskii2021asymptotically, koltchinskii2018asymptotic,koltchinskii2021estimation,koltchinskii2022estimation}, including spectral projectors \cite{koltchinskii2014asymptotics,koltchinskii2016asymptotics,koltchinskii2017normal,jirak2024quantitative}, functionals of principal components \cite{koltchinskii2017new,koltchinskii2020efficient}, and trace functionals and spectral measures of the covariance \cite{koltchinskii2024estimation}. An interesting direction for future research is to explore the estimation of functionals of higher-order moment tensors. We anticipate that Theorem \ref{thm:main1} could serve as a first step for further study in this direction.

Another interesting question concerns the \emph{rank} of random tensors. In this paper, we have studied the concentration properties of sums of i.i.d. simple (rank-one) tensors. For $p=2$, \cite[Theorem 1.2]{zhivotovskiy2024dimension} establishes a general-rank version of Koltchinskii-Lounici theorem using the variational
principle and the PAC-Bayesian method, which provides a dimension-free deviation bound for sums of independent positive semi-definite matrices that are not necessarily rank one. More recently, \cite{nakakita2024dimension2} obtained dimension-free bounds for sums of matrices and operators with dependence and arbitrary heavy tails. A natural and interesting direction for future research is to extend Theorem \ref{thm:main1} to the general-rank setting, deriving a concentration inequality for sums of i.i.d. symmetric tensors of arbitrary rank.

The recent work \cite{mikulincer2022clt} proves the following high-dimensional central limit theorem (CLT): Let $X, X_1,\ldots,X_N \iid \mu$ be random vectors in $\R^d$, where $\mu$ is either a symmetric uniform log-concave measure or a product measure. Then,
\[
\frac{1}{\sqrt{N}}\sum_{i=1}^N \left(X_i^{\otimes p}-\E\, X^{\otimes p}\right) \to \text{ Gaussian},
\]
provided that $N\gg d^{2p-1}$. Given the dimension-free nature of our operator norm bound in Theorem \ref{thm:main1}, an intriguing question is whether a dimension-free CLT can be established for the normalized sum of rank-one tensors.

For a sub-Gaussian function class $\F$, Theorem \ref{thm:main2} establishes a sharp concentration inequality for the supremum of the empirical process \[
\frac{1}{N}\sum_{i=1}^N |f|^p(X_i)-\E |f|^p(X)
\]
for any $p\ge 2$. This result improves existing bounds and lifts classical bounds on the suprema of quadratic empirical processes to higher-order settings. Additionally, as noted in Remark \ref{remark:general_p}, our analysis extends to the regime $p\in [3/2,2)$. A natural direction for future work is to further generalize these results to any $p>0$. More broadly, as discussed in \cite{mendelson2010empirical}, a fundamental objective in empirical process theory is to analyze the deviation 
\[
\frac{1}{N}\sum_{i=1}^N \ell(f(X_i))-\E\, \ell(f(X)),
\]
uniformly in $f\in \F$, where $\ell$ is a reasonable real-valued function and $\F$ is a suitable function class. Our results in this paper address this question for the cases $\ell(t)=t^p$ and $\ell(t)=|t|^p$ with any $p\ge 2$, when $\F$ is sub-Gaussian.

Another important and closely related line of research in empirical process theory concerns the study of \emph{multiplier} empirical processes \cite{ledoux2013probability,van2023weak}, given by
\[
f \mapsto \sum_{i=1}^N \xi_i f\left(X_i\right),
\]
where $f\in \F$ for some function class $\F$, and $(\xi_i)_{i=1}^{N}$ are random variables called ``multipliers''. We refer to \cite{mendelson2016upper,mendelson2017multiplier,han2019convergence, han2022multiplier} for recent developments. In particular, \cite{han2022multiplier}  developed theory and tools for analyzing multiplier $U$-processes of the form
\[
f \mapsto \sum_{1 \leq i_1<\cdots<i_m \leq N} \xi_{i_1} \cdots \xi_{i_m} f\left(X_{i_1}, \ldots, X_{i_m}\right),
\]
which is a natural higher-order generalization of the multiplier empirical processes. The framework developed in this paper provides a natural way for studying an additional generalization:
\[
(f_1,\ldots,f_{p})\mapsto \sum_{i=1}^{N} \xi_{i} f_1(X_i)\cdots f_p(X_i),
\]
where $f_k\in \F^{(k)} (1\le k\le p)$ for some function classes $(\F^{(k)})_{k=1}^{p}$. We leave a detailed exploration of this setting as an interesting direction for future research.

In this paper, we focused on Gaussian and sub-Gaussian random variables in developing tensor concentration inequalities, as well as on sub-Gaussian function classes $\F$ in the study of $L_p$ empirical processes. An important direction for future research is to extend these results to more general measures and weakly bounded empirical processes, such as those associated with isotropic log-concave measures \cite{guedon2007lp,mendelson2008weakly,mendelson2010empirical,mendelson2016upper} and Gaussian mixtures \cite[Section 4]{bandeira2021spectral}. We anticipate that our analysis can be extended to general distributions by further incorporating the powerful framework developed in \cite{mendelson2016upper}. As shown in \cite[Section 4.3]{mendelson2016upper}, the graded version of $\gamma$-type functionals offers advantages over the $\psi_2$-based complexity measure when analyzing unconditional log-concave ensembles. We expect these favorable properties to persist as we extend the results to higher-order settings.

Finally, we highlight several promising directions for future research in the broader field of concentration inequalities and statistical applications. An important open problem is to develop a sharp convex concentration inequality for random symmetric tensors, as proposed by Vershynin in \cite[Section 1.5]{vershynin2020concentration}. We expect that our tensor concentration inequality for sums of simple random tensors can offer useful insights toward addressing this challenge. In the literature on robust statistics, recent advances have focused on dimension-free bounds for mean and covariance estimation under adversarial corruption and heavy-tailed distributions \cite{lugosi2019mean,mendelson2020robust, abdalla2024covariance,oliveira2024improved,minasyan2023statistically}. Our results on dimension-free bounds for simple random tensors could serve as a foundation for developing dimension-free robust (moment) tensor estimation methods. Furthermore, we anticipate that Theorem \ref{thm:main1} may have applications in various tensor-related statistical problems \cite{mccullagh2018tensor,bi2021tensors,auddy2024tensors}, including tensor estimation \cite{han2022optimal,diakonikolas2024implicit}, tensor completion \cite{yuan2016tensor,xia2021statistically}, tensor regression \cite{luo2024tensor}, tensor singular value decomposition (SVD) \cite{zhang2018tensor,zhang2019optimal}, and tensor embeddings \cite{jiang2022near}. Additionally, our results may be relevant to independent component analysis (ICA), where fourth-order moment tensor (kurtosis) estimation plays a crucial role \cite{auddy2023large}. More broadly, our findings could contribute to advances in learning theory and empirical risk minimization \cite{even2021concentration}, and the method of moments \cite{sherman2020estimating}, with potential applications to cryo-EM and multi-reference alignment problems \cite{perry2019sample,bandeira2020optimal,dou2024rates}.


\section*{Funding}
This work was funded by NSF CAREER award NSF DMS-2237628.

\section*{Data availability}

No new data were generated or analyzed in support of this research.

\bibliographystyle{plain}
\bibliography{references}

\begin{thebibliography}{10}

\bibitem{abdalla2024covariance}
P.~Abdalla and N.~Zhivotovskiy.
\newblock {Covariance estimation: Optimal dimension-free guarantees for adversarial corruption and heavy tails}.
\newblock {\em Journal of the European Mathematical Society}, 2024.

\bibitem{adamczak2010quantitative}
R.~Adamczak, A.~Litvak, A.~Pajor, and N.~Tomczak-Jaegermann.
\newblock Quantitative estimates of the convergence of the empirical covariance matrix in log-concave ensembles.
\newblock {\em Journal of the American Mathematical Society}, 23(2):535--561, 2010.

\bibitem{auddy2024tensors}
A.~Auddy, D.~Xia, and M.~Yuan.
\newblock Tensors in high-dimensional data analysis: Methodological opportunities and theoretical challenges.
\newblock {\em Annual Review of Statistics and Its Application}, 12, 2024.

\bibitem{auddy2023large}
A.~Auddy and M.~Yuan.
\newblock {Large dimensional independent component analysis: Statistical optimality and computational tractability}.
\newblock {\em arXiv preprint arXiv:2303.18156}, 2023.

\bibitem{bamberger2022hanson}
S.~Bamberger, F.~Krahmer, and R.~Ward.
\newblock {The Hanson--Wright inequality for random tensors}.
\newblock {\em Sampling Theory, Signal Processing, and Data Analysis}, 20(2):14, 2022.

\bibitem{bandeira2021spectral}
A.~S. Bandeira and M.~T. Boedihardjo.
\newblock {The spectral norm of Gaussian matrices with correlated entries}.
\newblock {\em arXiv preprint arXiv:2104.02662}, 2021.

\bibitem{bandeira2023matrix}
A.~S. Bandeira, M.~T. Boedihardjo, and R.~van Handel.
\newblock Matrix concentration inequalities and free probability.
\newblock {\em Inventiones Mathematicae}, 234(1):419--487, 2023.

\bibitem{bandeira2024matrix}
A.~S. Bandeira, G.~Cipolloni, D.~Schr{\"o}der, and R.~van Handel.
\newblock Matrix concentration inequalities and free probability ii. two-sided bounds and applications.
\newblock {\em arXiv preprint arXiv:2406.11453}, 2024.

\bibitem{bandeira2024geometric}
A.~S. Bandeira, S.~Gopi, H.~Jiang, K.~Lucca, and T.~Rothvoss.
\newblock A geometric perspective on the injective norm of sums of random tensors.
\newblock {\em arXiv preprint arXiv:2411.10633}, 2024.

\bibitem{bandeira2020optimal}
A.~S. Bandeira, J.~Niles-Weed, and P.~Rigollet.
\newblock Optimal rates of estimation for multi-reference alignment.
\newblock {\em Mathematical Statistics and Learning}, 2(1):25--75, 2020.

\bibitem{bednorz2014}
W.~Bednorz.
\newblock {Concentration via chaining method and its applications}.
\newblock {\em arXiv preprint arXiv:1405.0676}, 2014.

\bibitem{bednorz2016bounds}
W.~Bednorz.
\newblock Bounds for stochastic processes on product index spaces.
\newblock In {\em High Dimensional Probability VII: The Carg{\`e}se Volume}, pages 327--357. Springer, 2016.

\bibitem{bi2021tensors}
X.~Bi, X.~Tang, Y.~Yuan, Y.~Zhang, and A.~Qu.
\newblock Tensors in statistics.
\newblock {\em Annual review of statistics and its application}, 8(1):345--368, 2021.

\bibitem{boedihardjo2024injective}
M.~T. Boedihardjo.
\newblock Injective norm of random tensors with independent entries.
\newblock {\em arXiv preprint arXiv:2412.21193}, 2024.

\bibitem{brailovskaya2024universality}
T.~Brailovskaya and R.~van Handel.
\newblock Universality and sharp matrix concentration inequalities.
\newblock {\em Geometric and Functional Analysis}, 34(6):1734--1838, 2024.

\bibitem{diakonikolas2024implicit}
I.~Diakonikolas and D.~M. Kane.
\newblock Implicit high-order moment tensor estimation and learning latent variable models.
\newblock {\em arXiv preprint arXiv:2411.15669}, 2024.

\bibitem{dirksen2015tail}
S.~Dirksen.
\newblock {Tail Bounds via Generic Chaining}.
\newblock {\em Electronic Journal of Probability}, 20:1--29, 2015.

\bibitem{dou2024rates}
Z.~Dou, Z.~Fan, and H.~H. Zhou.
\newblock Rates of estimation for high-dimensional multireference alignment.
\newblock {\em The Annals of Statistics}, 52(1):261--284, 2024.

\bibitem{even2021concentration}
M.~Even and L.~Massouli{\'e}.
\newblock Concentration of non-isotropic random tensors with applications to learning and empirical risk minimization.
\newblock {\em arXiv preprint arXiv:2102.04259}, 2021.

\bibitem{giannopoulos2000concentration}
A.~A. Giannopoulos and V.~D. Milman.
\newblock Concentration property on probability spaces.
\newblock {\em Advances in Mathematics}, 156(1):77--106, 2000.

\bibitem{gine1984some}
E.~Gin{\'e} and J.~Zinn.
\newblock Some limit theorems for empirical processes.
\newblock {\em The Annals of Probability}, pages 929--989, 1984.

\bibitem{gotze2021concentration}
F.~G{\"o}tze, H.~Sambale, and A.~Sinulis.
\newblock Concentration inequalities for polynomials in alpha-sub-exponential random variables.
\newblock {\em Electronic Journal of Probability}, 26, 2021.

\bibitem{guedon2007lp}
O.~Gu{\'e}don and M.~Rudelson.
\newblock Lp-moments of random vectors via majorizing measures.
\newblock {\em Advances in Mathematics}, 208(2):798--823, 2007.

\bibitem{han2022multiplier}
Q.~Han.
\newblock Multiplier u-processes: sharp bounds and applications.
\newblock {\em Bernoulli}, 28(1):87--124, 2022.

\bibitem{han2022exact}
Q.~Han.
\newblock Exact bounds for some quadratic empirical processes with applications.
\newblock {\em arXiv preprint arXiv:2207.13594v3}, 2024.

\bibitem{han2019convergence}
Q.~Han and J.~A. Wellner.
\newblock Convergence rates of least squares regression estimators with heavy-tailed errors.
\newblock {\em The Annals of Statistics}, 47(4):2286--2319, 2019.

\bibitem{han2022optimal}
R.~Han, R.~Willett, and A.~R. Zhang.
\newblock An optimal statistical and computational framework for generalized tensor estimation.
\newblock {\em The Annals of Statistics}, 50(1):1--29, 2022.

\bibitem{jiang2022near}
Q.~Jiang.
\newblock {Near-isometric properties of Kronecker-structured random tensor embeddings}.
\newblock {\em Advances in Neural Information Processing Systems}, 35:10191--10202, 2022.

\bibitem{jirak2024quantitative}
M.~Jirak and M.~Wahl.
\newblock Quantitative limit theorems and bootstrap approximations for empirical spectral projectors.
\newblock {\em Probability Theory and Related Fields}, 190(1):119--177, 2024.

\bibitem{klartag2005empirical}
B.~Klartag and S.~Mendelson.
\newblock Empirical processes and random projections.
\newblock {\em Journal of Functional Analysis}, 225(1):229--245, 2005.

\bibitem{koltchinskii2018asymptotic}
V.~Koltchinskii.
\newblock Asymptotic efficiency in high-dimensional covariance estimation.
\newblock In {\em Proceedings of the International Congress of Mathematicians: Rio de Janeiro 2018}, pages 2903--2923. World Scientific, 2018.

\bibitem{koltchinskii2021asymptotically}
V.~Koltchinskii.
\newblock Asymptotically efficient estimation of smooth functionals of covariance operators.
\newblock {\em Journal of the European Mathematical Society}, 23(3), 2021.

\bibitem{koltchinskii2022estimation}
V.~Koltchinskii.
\newblock {Estimation of smooth functionals in high-dimensional models: bootstrap chains and Gaussian approximation}.
\newblock {\em The Annals of Statistics}, 50(4):2386--2415, 2022.

\bibitem{koltchinskii2024estimation}
V.~Koltchinskii.
\newblock Estimation of trace functionals and spectral measures of covariance operators in gaussian models.
\newblock {\em arXiv preprint arXiv:2402.11321}, 2024.

\bibitem{koltchinskii2020efficient}
V.~Koltchinskii, M.~L{\"o}ffler, and R.~Nickl.
\newblock {Efficient estimation of linear functionals of principal components}.
\newblock {\em The Annals of Statistics}, 48(1):464--490, 2020.

\bibitem{koltchinskii2014asymptotics}
V.~Koltchinskii and K.~Lounici.
\newblock Asymptotics and concentration bounds for spectral projectors of sample covariance.
\newblock {\em arXiv preprint arXiv:1408.4643}, 2014.

\bibitem{koltchinskii2016asymptotics}
V.~Koltchinskii and K.~Lounici.
\newblock Asymptotics and concentration bounds for bilinear forms of spectral projectors of sample covariance.
\newblock In {\em Annales de l'Institut Henri Poincar{\'e} (B) Probabilit{\'e}s et Statistiques}, volume~52, 2016.

\bibitem{koltchinskii2017concentration}
V.~Koltchinskii and K.~Lounici.
\newblock Concentration inequalities and moment bounds for sample covariance operators.
\newblock {\em Bernoulli}, 23(1):110--133, 2017.

\bibitem{koltchinskii2017new}
V.~Koltchinskii and K.~Lounici.
\newblock New asymptotic results in principal component analysis.
\newblock {\em Sankhya A}, 79:254--297, 2017.

\bibitem{koltchinskii2017normal}
V.~Koltchinskii and K.~Lounici.
\newblock {Normal approximation and concentration of spectral projectors of sample covariance}.
\newblock {\em The Annals of Statistics}, 45(1):121--157, 2017.

\bibitem{koltchinskii2021estimation}
V.~Koltchinskii and M.~Zhilova.
\newblock Estimation of smooth functionals in normal models: bias reduction and asymptotic efficiency.
\newblock {\em The Annals of Statistics}, 49(5):2577--2610, 2021.

\bibitem{ledoux2013probability}
M.~Ledoux and M.~Talagrand.
\newblock {\em Probability in Banach spaces: isoperimetry and processes}.
\newblock Springer Science \& Business Media, 2013.

\bibitem{liaw2017simple}
C.~Liaw, A.~Mehrabian, Y.~Plan, and R.~Vershynin.
\newblock A simple tool for bounding the deviation of random matrices on geometric sets.
\newblock In {\em Geometric Aspects of Functional Analysis: Israel Seminar (GAFA) 2014--2016}, pages 277--299. Springer, 2017.

\bibitem{lugosi2019mean}
G.~Lugosi and S.~Mendelson.
\newblock Mean estimation and regression under heavy-tailed distributions: A survey.
\newblock {\em Foundations of Computational Mathematics}, 19(5):1145--1190, 2019.

\bibitem{luo2024tensor}
Y.~Luo and A.~R. Zhang.
\newblock Tensor-on-tensor regression: Riemannian optimization, over-parameterization, statistical-computational gap and their interplay.
\newblock {\em The Annals of Statistics}, 52(6):2583--2612, 2024.

\bibitem{mccullagh2018tensor}
P.~McCullagh.
\newblock {\em Tensor methods in statistics: Monographs on statistics and applied probability}.
\newblock Chapman and Hall/CRC, 2018.

\bibitem{mendelson2008weakly}
S.~Mendelson.
\newblock On weakly bounded empirical processes.
\newblock {\em Mathematische Annalen}, 340(2):293--314, 2008.

\bibitem{mendelson2010empirical}
S.~Mendelson.
\newblock Empirical processes with a bounded $\psi_1$ diameter.
\newblock {\em Geometric and Functional Analysis}, 20(4):988--1027, 2010.

\bibitem{mendelson2011discrepancy}
S.~Mendelson.
\newblock Discrepancy, chaining and subgaussian processes.
\newblock {\em The Annals of Probability}, 39(3):985--1026, 2011.

\bibitem{mendelson2016dvoretzky}
S.~Mendelson.
\newblock Dvoretzky type theorems for subgaussian coordinate projections.
\newblock {\em Journal of Theoretical Probability}, 29:1644--1660, 2016.

\bibitem{mendelson2016upper}
S.~Mendelson.
\newblock Upper bounds on product and multiplier empirical processes.
\newblock {\em Stochastic Processes and their Applications}, 126(12):3652--3680, 2016.

\bibitem{mendelson2017multiplier}
S.~Mendelson.
\newblock On multiplier processes under weak moment assumptions.
\newblock In {\em Geometric Aspects of Functional Analysis: Israel Seminar (GAFA) 2014--2016}, pages 301--318. Springer, 2017.

\bibitem{mendelson2021approximating}
S.~Mendelson.
\newblock {Approximating $L_p$ unit balls via random sampling}.
\newblock {\em Advances in Mathematics}, 386:107829, 2021.

\bibitem{mendelson2007reconstruction}
S.~Mendelson, A.~Pajor, and N.~Tomczak-Jaegermann.
\newblock Reconstruction and subgaussian operators in asymptotic geometric analysis.
\newblock {\em Geometric and Functional Analysis}, 17(4):1248--1282, 2007.

\bibitem{mendelson2012generic}
S.~Mendelson and G.~Paouris.
\newblock On generic chaining and the smallest singular value of random matrices with heavy tails.
\newblock {\em Journal of Functional Analysis}, 262(9):3775--3811, 2012.

\bibitem{mendelson2020robust}
S.~Mendelson and N.~Zhivotovskiy.
\newblock Robust covariance estimation under $l4- l2$ norm equivalence.
\newblock {\em The Annals of Statistics}, 48(3):1648--1664, 2020.

\bibitem{mikulincer2022clt}
D.~Mikulincer.
\newblock {A CLT in Stein’s distance for generalized Wishart matrices and higher-order tensors}.
\newblock {\em International Mathematics Research Notices}, 2022(10):7839--7872, 2022.

\bibitem{minasyan2023statistically}
A.~Minasyan and N.~Zhivotovskiy.
\newblock {Statistically optimal robust mean and covariance estimation for anisotropic Gaussians}.
\newblock {\em arXiv preprint arXiv:2301.09024}, 2023.

\bibitem{minsker2017some}
S.~Minsker.
\newblock {On some extensions of Bernstein’s inequality for self-adjoint operators}.
\newblock {\em Statistics \& Probability Letters}, 127:111--119, 2017.

\bibitem{montgomery1990distribution}
S.~J. Montgomery-Smith.
\newblock {The distribution of Rademacher sums}.
\newblock {\em Proceedings of the American Mathematical Society}, 109(2):517--522, 1990.

\bibitem{nakakita2024dimension2}
S.~Nakakita, P.~Alquier, and M.~Imaizumi.
\newblock Dimension-free bounds for sums of dependent matrices and operators with heavy-tailed distributions.
\newblock {\em Electronic Journal of Statistics}, 18(1):1130--1159, 2024.

\bibitem{nemirovski2004interior}
A.~Nemirovski.
\newblock Interior point polynomial time methods in convex programming.
\newblock {\em Lecture Notes}, 42(16):3215--3224, 2004.

\bibitem{oliveira2024improved}
R.~I. Oliveira and Z.~F. Rico.
\newblock {Improved covariance estimation: optimal robustness and sub-Gaussian guarantees under heavy tails}.
\newblock {\em The Annals of Statistics}, 52(5):1953--1977, 2024.

\bibitem{perry2019sample}
A.~Perry, J.~Niles-Weed, A.~S. Bandeira, P.~Rigollet, and A.~Singer.
\newblock The sample complexity of multireference alignment.
\newblock {\em SIAM Journal on Mathematics of Data Science}, 1(3):497--517, 2019.

\bibitem{puchkin2025sharper}
N.~Puchkin, F.~Noskov, and V.~Spokoiny.
\newblock Sharper dimension-free bounds on the frobenius distance between sample covariance and its expectation.
\newblock {\em Bernoulli}, 31(2):1664--1691, 2025.

\bibitem{reed1980methods}
M.~Reed and B.~Simon.
\newblock {\em {Methods of Modern Mathematical Physics: Functional Analysis}}, volume~1.
\newblock Gulf Professional Publishing, 1980.

\bibitem{rudelson1999random}
M.~Rudelson.
\newblock Random vectors in the isotropic position.
\newblock {\em Journal of Functional Analysis}, 164(1):60--72, 1999.

\bibitem{santos2023almost}
P.~O. Santos.
\newblock {Almost sharp covariance and Wishart-type matrix estimation}.
\newblock {\em arXiv preprint arXiv:2307.09190}, 2023.

\bibitem{sherman2020estimating}
S.~Sherman and T.~G. Kolda.
\newblock Estimating higher-order moments using symmetric tensor decomposition.
\newblock {\em SIAM Journal on Matrix Analysis and Applications}, 41(3):1369--1387, 2020.

\bibitem{talagrand2022upper}
M.~Talagrand.
\newblock {\em Upper and Lower Bounds for Stochastic Processes: Decomposition Theorems}, volume~60.
\newblock Springer Nature, 2022.

\bibitem{tikhomirov2018sample}
K.~Tikhomirov.
\newblock Sample covariance matrices of heavy-tailed distributions.
\newblock {\em International Mathematics Research Notices}, 2018(20):6254--6289, 2018.

\bibitem{tropp2015introduction}
J.~A. Tropp.
\newblock An introduction to matrix concentration inequalities.
\newblock {\em Foundations and Trends{\textregistered} in Machine Learning}, 8(1-2):1--230, 2015.

\bibitem{tropp2016expected}
J.~A. Tropp.
\newblock The expected norm of a sum of independent random matrices: an elementary approach.
\newblock In {\em High Dimensional Probability VII: The Cargese Volume}, pages 173--202. Springer, 2016.

\bibitem{van2023weak}
A.~W. van~der Vaart and J.~A. Wellner.
\newblock {\em Weak Convergence and Empirical Processes: With Applications to Statistics}.
\newblock Springer Nature, 2023.

\bibitem{van2017structured}
R.~Van~Handel.
\newblock Structured random matrices.
\newblock {\em Convexity and Concentration}, pages 107--156, 2017.

\bibitem{vershynin2010introduction}
R.~Vershynin.
\newblock Introduction to the non-asymptotic analysis of random matrices.
\newblock {\em arXiv preprint arXiv:1011.3027}, 2010.

\bibitem{vershynin2011approximating}
R.~Vershynin.
\newblock Approximating the moments of marginals of high-dimensional distributions.
\newblock {\em The Annals of Probability}, pages 1591--1606, 2011.

\bibitem{vershynin2018high}
R.~Vershynin.
\newblock {\em High-Dimensional Probability: An Introduction with Applications in Data Science}, volume~47.
\newblock Cambridge University Press, 2018.

\bibitem{vershynin2020concentration}
R.~Vershynin.
\newblock Concentration inequalities for random tensors.
\newblock {\em Bernoulli}, 26(4):3139--3162, 2020.

\bibitem{xia2021statistically}
D.~Xia, M.~Yuan, and C.-H. Zhang.
\newblock Statistically optimal and computationally efficient low rank tensor completion from noisy entries.
\newblock {\em The Annals of Statistics}, 49(1):76--99, 2021.

\bibitem{yuan2016tensor}
M.~Yuan and C.-H. Zhang.
\newblock On tensor completion via nuclear norm minimization.
\newblock {\em Foundations of Computational Mathematics}, 16(4):1031--1068, 2016.

\bibitem{zhang2019optimal}
A.~Zhang and R.~Han.
\newblock Optimal sparse singular value decomposition for high-dimensional high-order data.
\newblock {\em Journal of the American Statistical Association}, 2019.

\bibitem{zhang2018tensor}
A.~Zhang and D.~Xia.
\newblock {Tensor SVD: Statistical and computational limits}.
\newblock {\em IEEE Transactions on Information Theory}, 64(11):7311--7338, 2018.

\bibitem{zhivotovskiy2024dimension}
N.~Zhivotovskiy.
\newblock Dimension-free bounds for sums of independent matrices and simple tensors via the variational principle.
\newblock {\em Electronic Journal of Probability}, 29:1--28, 2024.

\bibitem{zhou2021sparse}
Z.~Zhou and Y.~Zhu.
\newblock {Sparse random tensors: Concentration, regularization and applications}.
\newblock {\em Electronic Journal of Statistics}, 15(1), 2021.

\end{thebibliography}

\renewcommand{\theHsection}{A\arabic{section}}

\begin{appendix}

\section*{Appendix}
 
\section{Proofs of auxiliary results in Section \ref{subsec:main2}}\label{app:A}

\begin{proof}[Proof of Lemma \ref{lemma:gamma_2_d_psi_2}]

Let $ \tdiam_{\psi_2}(\mcF):=\sup_{f,g\in \F} \|f-g\|_{\psi_2}$. Note first that 
\begin{align*}
    \gamma(\mcF, \psi_2) 
    &= \inf \sup_{f\in \mcF} \sum_{s \ge 0} 2^{s/2} \|f-\pi_sf\|_{\psi_2}
   \overset{\text{($\star$)}}{\ge}
    \inf \sup_{f\in \mcF} \|f-\pi_0 f\|_{\psi_2}\\
    &\overset{\text{($\star\star$)}}{=}
    \inf_{g\in \mcF} \sup_{f\in \mcF} \|f-g\|_{\psi_2}
    \overset{\text{($\star\star\star$)}}{\ge}
    \frac{1}{2} \tdiam_{\psi_2}(\mcF),
\end{align*}
where ($\star$) follows by the positivity of the terms in the sum for $s \ge 1$, ($\star\star$) holds by the fact that the infimum is over all admissible sequences with $|\mcF_0|=1$, and ($\star\star\star$) holds by a standard bound on $\inf_{g\in \mcF} \sup_{f\in \mcF} \|f-g\|_{\psi_2}$, which is referred to in the literature as the Chebyshev radius of $\mcF$. Now, in the symmetric case, we have 
\begin{align*}
    d_{\psi_2}(\mcF) 
    =\frac{1}{2}
    \sup_{f \in \mcF} \| 2f\|_{\psi_2}
    =\frac{1}{2}
    \sup_{f \in \mcF} \| f-(-f)\|_{\psi_2}
    \le \frac{1}{2}\sup_{f,g \in \mcF} \| f-g\|_{\psi_2}
    = \frac{1}{2} \tdiam_{\psi_2}(\mcF).
\end{align*}
In the case that $0 \in \mcF$
\begin{align*}
    d_{\psi_2}(\mcF) 
    = \sup_{f \in \mcF} \| f - 0\|_{\psi_2}
    \le 
    \sup_{f,g \in \mcF} \| f-g\|_{\psi_2}
    =\tdiam_{\psi_2}(\mcF).
\end{align*}
Therefore, in either case we have $\gamma(\mcF,\psi_2) \gtrsim d_{\psi_2}(\mcF).$
\end{proof}

\begin{proof}[Proof of Lemma \ref{lemma:simple_tail}]

Without loss of generality, we assume $C_1(p)\ge  \sqrt{\frac{2 \log C_2(p)}{c_3(p)}}\lor 1$. For $u\ge C_1(p)$, let 
\[
   v:=\sqrt{\frac{c_3(p)u^2\omega-\log C_2(p)}{\omega}}.
   \]
Then, $ \exp(-v^2 \omega)=C_2(p)\exp(-c_3(p)u^2\omega)$. Observe that
   \[
   v= \sqrt{c_3(p)u^2-\frac{\log C_2(p)}{\omega}}\ge \sqrt{\frac{c_3(p)}{2}u^2}=\sqrt{\frac{c_3(p)}{2}}u,
   \]
   where the inequality follows since $\frac{1}{2}c_3(p)u^2\ge \frac{1}{2}c_3(p) C_1^2(p)\ge \log C_2(p)\ge \frac{\log C_2(p)}{\omega} $.
   Therefore, for any $v\ge \sqrt{c_3(p)C_1^2(p)-\frac{\log C_2(p)}{\omega}}=:C_4(p,\omega)$, it holds with probability at least $1-\exp(-v^2 \omega)$ that
   \[
   (*)\lesssim_p uA+u^p B \lesssim_p vA+v^p B. 
   \]
If $C_4(p,\omega)\le 1$, the proof is complete. Otherwise, if $C_4(p,\omega)>1$, we further observe that $C_4(p,\omega)\le \sqrt{c_3(p)C_1^2(p)}$ since $\frac{\log C_2(p)}{\omega}\ge 0$. In this case, for $1\le v < C_4(p,\omega)$, it holds with probability at least $1-\exp(-C^2_4(p,\omega) \omega)\ge 1-\exp(-v^2\omega)$ that
   \[
   (*)\lesssim_p (vA+v^p B)\big|_{v=C_4(p,\omega)}\le (vA+v^p B)\big|_{v=\sqrt{c_3(p)C_1^2(p)}}\lesssim_p A+B\lesssim_p vA+v^p B.
   \]
   Combining the cases $v\in [C_4(p,\omega),\infty)$ and $v\in [1,C_4 (p,\omega))$ completes the proof.
\end{proof}
 
 \section{Proofs of auxiliary results in Section \ref{subsec:l_m_norm}}\label{app:B}

\begin{proof}[Proof of Lemma \ref{lemma:property_phi}]
Observe that
    \begin{align*}
    \phi^{-1}(xy)&=\max\left\{\sqrt{\log_2(1+xy)},\sqrt{\frac{(\log_2(1+xy))^{m}}{N}}\right\}\\
    &\le \sqrt{\log_2(1+xy)}+\sqrt{\frac{(\log_2(1+xy))^{m}}{N}}\\
    &\overset{\text{($\star$)}}{\le} \sqrt{\log_2(1+x)}+\sqrt{\log_2(1+y)}+\sqrt{\frac{\left(\log_2(1+x)+\log_2(1+y)\right)^{m}}{N}}\\
    &\lesssim_m \sqrt{\log_2(1+x)}+\sqrt{\log_2(1+y)}+\sqrt{\frac{(\log_2(1+x))^{m}}{N}}+\sqrt{\frac{(\log_2(1+y))^{m}}{N}}\\
    &\lesssim \phi^{-1}(x)+\phi^{-1}(y),
    \end{align*}
where ($\star$) follows by $(1+xy)\le (1+x)(1+y)$ and $\sqrt{\log_2(1+x)+\log_2(1+y)}\le \sqrt{\log_2(1+x)}+\sqrt{\log_2(1+y)}$.
\end{proof}

\begin{proof}[Proof of Lemma \ref{lemma:alpha-subexp-2}]

We will use the following $\alpha$-sub-exponential concentration inequality from \cite{gotze2021concentration}. For $\alpha>0,$ we define the Orlicz (quasi-norm) of a random variable $X$ by
\[
\|X\|_{\psi_\alpha}:=\inf \left\{c>0: 
\mathbb{E}_{X \sim \mu}\insquare{ \exp \inparen{\frac{|X|^\alpha}{c^\alpha}}} \leq 2\right\}.
\]
\begin{lemma}[{$\alpha$-sub-exponential concentration, \cite[Corollary 1.4]{gotze2021concentration}}]\label{lemma:sub-exponential-1}
    Let $X_1,\dots, X_N$ be independent, centered random variables with $\|X_i\|_{\psi_\alpha} \le M$ for $\alpha \in (0,1].$ Then for $a \in \R^N$ and any $u \ge 0$, 
    \begin{align*}
        \P\bigg(\Big|\sum_{i=1}^N a_i X_i \Big|\ge u\bigg) \le 2 \exp 
        \left(
        -\frac{1}{C_\alpha} \min\bigg\{
        \frac{u^2}{M^2\|a\|_{\ell_2}^2}
        ,
        \frac{u^\alpha}{M^\alpha \max_i |a_i|^\alpha}
        \bigg\}
        \right).
    \end{align*}
\end{lemma}

We apply Lemma \ref{lemma:sub-exponential-1} with $\alpha=2/m$ and $a_1=a_2=\cdots =a_N=1/\sqrt{N}$ to obtain that: if $Y_1,\dots, Y_N$ are independent random variables with $\|Y_i-\E Y_i\|_{\psi_{2/m}} \le M$ for every $1\le i\le N$, then for any $u> 0$,
\begin{align}\label{eq:prop_lemma_aux0}
\mathbb{P}\bigg(\frac{|\sum_{i=1}^N Y_i-\E Y_i|}{c(m)\sqrt{N}M}\ge u\bigg)\le 
2\exp(-5\min\{u^2,(\sqrt{N}u)^{2/m}\})\le \frac{2}{\phi(u)^2},
\end{align}
where $c>1$ is some constant depending only on $m$ and $\phi(u)=2^{\min\{(\sqrt{N}u)^{2/m},\, u^2\}}-1.$ Therefore,
\begin{align}\label{eq:prop_lemma_aux1}
\E \phi\bigg(\frac{|\sum_{i=1}^N Y_i-\E Y_i|}{c(m)\sqrt{N}M}\bigg)&=\int_0^{\infty} \mathbb{P}\bigg(\phi\bigg(\frac{|\sum_{i=1}^N Y_i-\E Y_i|}{c(m)\sqrt{N}M}\bigg)\ge u\bigg)du \nonumber\\
&=\int_0^{\infty} \mathbb{P}\bigg(\frac{|\sum_{i=1}^N Y_i-\E Y_i|}{c(m)\sqrt{N}M}\ge \phi^{-1}(u)\bigg)du \nonumber\\
&\le 1+\int_{1}^{\infty}\mathbb{P}\bigg(\frac{|\sum_{i=1}^N Y_i-\E Y_i|}{c(m)\sqrt{N}M}\ge \phi^{-1}(u)\bigg)du \nonumber\\
&\overset{\eqref{eq:prop_lemma_aux0}}{\le} 1+\int_{1}^{\infty}\frac{2}{u^2}du=3.
\end{align}
    
For any $g,h\in \F$, using that $\psi_{2/m}$ is a quasi-norm \cite[Lemma A.1, A.3]{gotze2021concentration} we have
\[
\||g-h|^{m}\|_{\psi_{2/m}}\lesssim_m \|g-h\|_{\psi_2}^m,
\]
and 
\begin{align}\label{eq:prop_lemma_aux2}
\||g-h|^{m}-\E |g-h|^{m}\|_{\psi_{2/m}}\lesssim_m \||g-h|^{m}\|_{\psi_{2/m}}\lesssim_m \|g-h\|_{\psi_2}^m.
\end{align}
Similarly,
    \begin{align}\label{eq:prop_lemma_aux3}
    &\||g|^m-|h|^m-\E[|g|^m-|h|^m]\|_{\psi_{2/m}}\lesssim_{\alpha} \||g|^m-|h|^m\|_{\psi_{2/m}} \nonumber\\
    &=\bigg\|(|g|-|h|)\bigg(\sum_{\ell=0}^{m-1} |g|^{\ell} |h|^{m-1-\ell} \bigg)\bigg\|_{\psi_{2/m}} \nonumber\\
    &\lesssim_m \||g|-|h|\|_{\psi_2}\bigg\|\sum_{\ell=0}^{m-1} |g|^{\ell} |h|^{m-1-\ell}\bigg\|_{\psi_{2/(m-1)}} \nonumber\\
    &\lesssim_m \|g-h\|_{\psi_2} \left(\sum_{\ell=0}^{m-1} \|g\|^{\ell}_{\psi_2}\|h\|^{m-1-\ell}_{\psi_2}\right) \nonumber\\
    &\lesssim_m \|g-h\|_{\psi_2}d_{\psi_2}^{m-1}(\F).
    \end{align}
We also have 
\begin{align}\label{eq:prop_lemma_aux4}
\||g|^m\|_{\psi_{2/m}}\lesssim_m \|g\|_{\psi_2}^m\le d_{\psi_2}^m(\F).
\end{align}
    The proof is completed by applying \eqref{eq:prop_lemma_aux1} with $Y_i=|g-h|^m(X_i), 1\le i\le N$ and the upper bound of $\psi_{2/m}$ quasi-norm \eqref{eq:prop_lemma_aux2}, applying \eqref{eq:prop_lemma_aux1} with $Y_i=(|g|^m-|h|^m)(X_i), 1\le i\le N$ and the upper bound \eqref{eq:prop_lemma_aux3}, and applying \eqref{eq:prop_lemma_aux1} with $Y_{i}=|g|^m(X_i), 1\le i\le N$ and the upper bound \eqref{eq:prop_lemma_aux4}.
\end{proof}

\end{appendix}

\end{document}